\documentclass[reqno]{amsart}

\newcommand\version{April 06, 2022}

\usepackage{amsmath,amsfonts,amsthm,amssymb,amsxtra}
\usepackage{hyperref}
\usepackage[english]{babel}
\usepackage{bbm} 
\usepackage{bm} 
\usepackage{color}

\newtheorem{theorem}{Theorem}[section]
\newtheorem{proposition}[theorem]{Proposition}
\newtheorem{lemma}[theorem]{Lemma}
\newtheorem{corollary}[theorem]{Corollary}

\theoremstyle{definition}

\newtheorem{definition}[theorem]{Definition}
\newtheorem{example}[theorem]{Example}

\theoremstyle{remark}

\newtheorem{remark}[theorem]{Remark}
\newtheorem{remarks}[theorem]{Remarks}

\numberwithin{equation}{section}

\newcommand{\C}{\mathbb{C}}

\newcommand{\DG}{\mathrm{DG}}
\renewcommand{\epsilon}{\varepsilon}

\newcommand{\loc}{{\rm loc}}

\newcommand{\N}{\mathbb{N}}

\renewcommand{\phi}{\varphi}
\newcommand{\R}{\mathbb{R}}

\newcommand{\Z}{\mathbb{Z}}

\DeclareMathOperator{\im}{Im}

\DeclareMathOperator{\re}{Re}

\DeclareMathOperator{\supp}{supp}

\def\cl{\mathcal{L}}

\newcommand{\me}[1]{\mathrm{e}^{#1}}
\newcommand{\one}{\mathbf{1}}

\makeatletter
\newcommand*{\rom}[1]{\expandafter\@slowromancap\romannumeral #1@}
\makeatother

\begin{document}

\title[Complex-time heat kernels via Phragm\'en--Lindel\"of --- \version]{On complex-time heat kernels of fractional Schr\"odinger operators via Phragm\'en--Lindel\"of principle}

\author{Konstantin Merz}
\address[Konstantin Merz]{Institut f\"ur Analysis und Algebra, Technische Universit\"at Braunschweig, Universit\"atsplatz 2, 38106 Braunschweig, Germany}
\email{k.merz@tu-bs.de}

\subjclass[2010]{Primary: 35K08; Secondary: 35J10}
\keywords{Complex-time heat kernel, Phragm\'en--Lindel\"of principle, fractional Laplace, fractional Schr\"odinger operator, Davies--Gaffney estimate}

\date{\version}

\begin{abstract}
  We consider fractional Schr\"odinger operators with possibly singular potentials and derive certain spatially averaged estimates for its complex-time heat kernel. The main tool is a Phragm\'en--Lindel\"of theorem for polynomially bounded functions on the right complex half-plane. The interpolation leads to possibly nonoptimal off-diagonal bounds.
\end{abstract}

\maketitle

\vspace{-1em}
\section{Introduction}

Phragm\'en--Lindel\"of theorems
(cf.~\cite{BakNewman2010,Hille1962,Rudin1987,SteinShakarchi2003C})
are powerful complex analysis tools that extend the maximum modulus
principle to certain unbounded domains.
They are often used in the presence of exponential bounds, which are,
e.g., available in the analysis of heat kernels of elliptic
second-order differential operators.
Suppose, e.g., that $\me{-tH}$ is a symmetric Markov semigroup on
$L^2(\R^d)$ whose integral kernel satisfies
\begin{align}
  \label{eq:gaussian}
  0 \leq \me{-tH}(x,y) \leq c t^{-\frac d2}\exp\left[-\frac{b|x-y|^2}{t}\right]\,,
  \quad t>0,\,x,y\in\R^d\,,
\end{align}
where $b,c$ are positive constants.
In \cite[Theorem~3.4.8]{Davies1990} and \cite[Lemma~9, Theorem~10]{Davies1995}
Davies used the Phragm\'en--Lindel\"of principle
to show that for all $\epsilon>0$ there is $c_\epsilon>0$ such that
\begin{align}
  \label{eq:complexgaussian}
  |\me{-zH}(x,y)| \leq c_\epsilon(\re z)^{-\frac d2}\exp\left[-\re\left(\frac{b|x-y|^2}{(1+\epsilon)z}\right)\right]\,,
  \quad x,y\in\R^d
\end{align}
holds for all complex times $z\in\C_+:=\{z\in\C:\,\re(z)>0\}$, whenever
\eqref{eq:gaussian} holds.

Coulhon and Sikora \cite[Section~4]{CoulhonSikora2008} reversed Davies'
ideas and used the Phragm\'en--Lindel\"of principle to derive 
gaussian heat kernel estimates like \eqref{eq:complexgaussian}
from suitably weighted off-diagonal estimates -- so-called 
Davies--Gaffney estimates (cf.~\cite[(3.2)]{CoulhonSikora2008} or \eqref{eq:plgaussian} for $\theta\! = \!0$ and $\alpha\! = \!2$) --
and on-diagonal estimates like $\me{-tH}(x,x) \!\leq\! c t^{-\frac{d}{2}}$.
The latter can often be derived from Sobolev inequalities corresponding
to $H$ by Nash's method, cf.~Davies \cite[Section~2.4]{Davies1990} and
Milman and Semenov \cite{MilmanSemenov2004}.

Complex-time heat kernel estimates like \eqref{eq:complexgaussian} are
of paramount importance in many problems in harmonic analysis and partial
differential equations, e.g.,
in proving $L^p$ boundedness of spectral multipliers and convergence of
Riesz means, and investigating maximal regularity properties of the
Schr\"odinger evolution $\me{itH}$ for operators $H$ whose heat kernels
satisfy sub-gaussian estimates, see, e.g.,
\cite{BlunckKunstmann2002,BlunckKunstmann2003,Blunck2003,BlunckKunstmann2004,BlunckKunstmann2005,Blunck2007,Carronetal2002,Chenetal2016,Chenetal2020S,Chenetal2022,CoulhonDuong1999,CoulhonDuong2000,CoulhonDuong2001,DuongRobinson1996,Duongetal2002,Hebisch1990,Kriegler2014,Sikoraetal2014}.

On the other hand, there has been recent interest in the fractional
Laplace operator $(-\Delta)^{\alpha/2}$ in $L^2(\R^d)$ for $\alpha>0$
and its generated holomorphic semigroup $\me{-z(-\Delta)^{\alpha/2}}$
for $z\in\C_+$.
For $\alpha>0$ Blumenthal and Getoor \cite{BlumenthalGetoor1960} proved
\begin{align}
  \label{eq:heatkernel}
  |\me{-t(-\Delta)^{\alpha/2}}(x)| \leq c_{d,\alpha} \frac{t}{(t^{1/\alpha}+|x|)^{d+\alpha}}\,, \quad t>0,\ x\in\R^d\,,
\end{align}
whereas for $\alpha=1$ one has the representation via Poisson's kernel
\begin{align*}
  \me{-t(-\Delta)^{1/2}}(x) = c_d\,\frac{t}{(t^2+|x|^2)^{(d+1)/2}}
\end{align*}
for an explicit constant $c_d>0$.
The slow decay of these kernels and their extension to $\C_+$ -- that we
discuss momentarily -- complicates many arguments in problems where 
heat kernel estimates are the central element of the analysis like
\cite{Franketal2021,Merz2021,BuiDAncona2021}. In particular, they impede
the generalization of the above mentioned works on spectral multipliers for
operators whose heat kernels only decay algebraically.
As \eqref{eq:complexgaussian} indicates, complex-time heat kernels satisfy
worse bounds as $|\arg(z)|$ increases. Thus, the derivation of sharp
estimates becomes even more crucial in this scenario.

Zhao and Zheng \cite[Theorem~1.3]{ZhaoZheng2020} recently proved uniform
complex-time heat kernel estimates for $(-\Delta)^{\alpha/2}$ and all
$\alpha>0$ using the stationary phase method.
While for $\alpha=1$ one has the explicit formula
\begin{align}
  \label{eq:poissonexplicit}
  \me{-z(-\Delta)^{1/2}}(x)
  = c_d\frac{z}{(z^2+|x|^2)^{(d+1)/2}}\,, \quad z\in\C_+,\, x\in\R^d\,,
\end{align}
the derivation of such estimates for $\alpha\neq1$ is rather intricate.
Pointwise estimates in the case $|\arg(z)|=\pi/2$ are however well known,
see, e.g., Miyachi \cite[Proposition~5.1]{Miyachi1981},
Wainger \cite[pp.~41-52]{Wainger1965},
Huang et al.~\cite{Huangetal2017}
and \cite[pp.~2-3]{ZhaoZheng2020}.

Using perturbation theory, Zhao and Zheng extended their results
and derived uniform complex-time heat kernel
estimates for fractional Schr\"odinger operators
\begin{align}
  \label{eq:defhalpha}
  H_\alpha := (-\Delta)^{\alpha/2} + V \quad \text{in}\ L^2(\R^d)
\end{align}
when $V\in L_\loc^1(\R^d:\R)$ belongs to the higher order Kato class
$K_\alpha(\R^d)$ (cf.~\cite{DaviesHinz1998,Gulisashvili2002,YaoZheng2009}
for a precise definition). In this case $V$ is infinitesimally form bounded
with respect to $(-\Delta)^{\alpha/2}$ and $H_\alpha$ can be defined as the
self-adjoint Friedrichs extension of the corresponding quadratic form
with form core $C_c^\infty(\R^d)$.
Their estimates for the kernel of the holomorphic extension of
$\me{-tH_\alpha}$ to $\C_+$ read as follows.

\begin{theorem}[{\cite[Theorem~1.5]{ZhaoZheng2020}}]
  \label{zhaozheng}
  Let $\alpha>0$, $z=|z|\me{i\theta}\in\C_+$, and $V\in K_\alpha(\R^d)$.
  Then for any $\epsilon\in(0,1)$ there are $c_{d,\alpha}>0$ and
  $\mu_{\epsilon,V,d,\alpha}>0$ such that
  \begin{align}
    \label{eq:zhaozheng}
    \begin{split}
      & |\me{-zH_\alpha}(x,y)|\\
      & \quad \leq c_{d,\alpha} \me{\mu_{\epsilon,V,d,\alpha}|z|} (\cos\theta)^{-d(\frac1\alpha-\frac12)\one_{\{\alpha<1\}} -(\frac{d}{2}+\alpha-1)\one_{\{\alpha\geq1\}}}\frac{|z|}{(|z|^{\frac1\alpha}+|x-y|)^{d+\alpha}}
    \end{split}
  \end{align}
  holds for all $x,y\in\R^d$.
\end{theorem}
Estimate \eqref{eq:zhaozheng} reflects the best possible off-diagonal
pointwise decay for $|x-y|\gg|z|^{1/\alpha}$ and all $|\theta|<\pi/2$
(compare with \eqref{eq:heatkernel}) and is particularly useful for
$|z|\lesssim1$.
The parameter $\mu_{\epsilon,V,d,\alpha}$ is defined in \cite[p.~22]{ZhaoZheng2020}.
Due to its complicated dependence on $\epsilon$, an optimization with respect
to $\epsilon\in(0,1)$ does not seem to be straightforward.
Moreover, their proof gives $\mu_{\epsilon,V,d,\alpha}>0$ even if $V\geq0$.

\medskip
In this note we derive a Phragm\'en--Lindel\"of principle for polynomially
bounded functions (Theorem \ref{PL}) and apply it to obtain suitably weighted
and averaged estimates for $\me{-zH_\alpha}$ that are not uniform in $\theta$
but do not deteriorate for $|z|\gg1$ (Subsections \ref{ss:posv}--\ref{ss:negv}).
Moreover, we can allow for critically singular potentials $V$, like the Hardy
potential $|x|^{-\alpha}$, which do not belong to $K_\alpha(\R^d)$.
In this case $\me{-tH_\alpha}$ is generally not $L^1\to L^\infty$ bounded.
On the downside, it is not clear whether the off-diagonal decay of our
estimates is optimal, especially when $|\pi/2-|\theta||\ll1$.

\subsection*{Organization and notation}
In Section~\ref{s:pl} we state and prove a Phragm\'en--Lindel\"of
theorem for polynomially bounded functions.
We apply this theorem in Section~\ref{s:application} to derive
estimates for the complex-time heat kernel of fractional Schr\"odinger
operators with non-negative potentials (Corollary~\ref{PLapplied} and
Corollary~\ref{PLapplied2} for a pointwise estimate), and
potentials with a non-vanishing, possibly singular, negative part
(Theorem~\ref{plgge} and Corollaries~\ref{plggecor}--\ref{lpboundednesscomplex}).
The applicability of these bounds is discussed in Subsection \ref{ss:applicability}.
In Section~\ref{ss:consequencesgge} we collect some properties
of the estimates -- so-called dyadic Davies--Gaffney estimates
(Definitions~\ref{gge} and \ref{ggerestricted}) --
that we discuss when $V$ has a negative part.

We write $A\lesssim B$ for two non-negative quantities $A,B\geq0$ to indicate
that there is a constant $C>0$ such that $A\leq C B$. If $C=C_\tau$ depends on
a parameter $\tau$, we write $A\lesssim_\tau B$.
The dependence on fixed parameters like $d$ and $\alpha$ is sometimes omitted.
The notation $A\sim B$ means $A\lesssim B\lesssim A$.
All constants are denoted by $c$ or $C$ and are allowed to change from line
to line.
We abbreviate $A\wedge B:=\min\{A,B\}$ and $A\vee B:=\max\{A,B\}$.
The Heaviside function is denoted by $\theta(x)$. We use the convention
$\theta(0)=1$.
The indicator function and the Lebesgue measure of a set $\Omega\subseteq\R^d$
are denoted by $\one_\Omega$ and $|\Omega|$, respectively.
The euclidean distance between two sets $\Omega_1,\Omega_2\subseteq\R^d$ is
denoted by $d(\Omega_1,\Omega_2):=\inf_{x\in\Omega_1,y\in\Omega_2}|x-y|$.
If $T:L^p(\R^d)\to L^q(\R^d)$ is a bounded linear operator, we write
$T\in\mathcal{B}(L^p\to L^q)$ and denote its operator norm by $\|T\|_{p\to q}$.
For $1\leq p\leq\infty$ we write $p'=(1-1/p)^{-1}$.

\section{Phragm\'en--Lindel\"of principle with polynomial bounds}
\label{s:pl}

\begin{theorem} 
  \label{PL}
  Let $X$ be a Banach space equipped with a norm $\|\cdot\|$ and
  $F:\C_+=\{z\in\C:\,\re(z)>0\}\to X$ be a holomorphic function satisfying
  \begin{align}
    \label{eq:assum1}
    \|F(|z|\me{i\theta})\|
    & \leq a_1(|z|\cos\theta)^{-\beta_1}
      \quad \text{and}\\
    \label{eq:assum2}
    \|F(|z|)\|
    & \leq a_1|z|^{-\beta_1}\left(\frac{a_2}{|z|}\right)^{-\beta_2}\cdot \left(\frac{a_3}{|z|}\right)^{\beta_3}
  \end{align}
  for some $a_1,a_2,a_3>0$,
  $\beta_1,\beta_2,\beta_3\geq0$,
  all $|z|>0$, and all $|\theta|<\pi/2$.
  Then for all $\epsilon\in(0,1)$ one has
  \begin{align}
    \label{eq:pl1}
    \begin{split}
      \|F(|z|\me{i\theta})\|
      & \leq a_1 (|z|\cos\theta)^{-\beta_1}
      \left[ 1 \wedge \epsilon^{-\beta_1}\left( \left(\frac{a_2}{|z|}\right)^{-\beta_2} \cdot \left(\frac{a_3}{|z|}\right)^{\beta_3}\right)^{1-|\theta|/\gamma(\epsilon,\theta)}\right]
    \end{split}
  \end{align}
  for all $|\theta|<\pi/2$ and $|z|>0$, where
  $\gamma(\epsilon,\theta):=\epsilon|\theta|+(1-\epsilon)\pi/2$.
\end{theorem}

\begin{proof}
  For $\gamma\in(0,\pi/2)$ and $z=|z|\me{i\theta}$ let
  \begin{align}
    \label{eq:defg}
    G(z) := z^{-\beta_1}F(z^{-1}) \cdot H_{2}(z) \cdot H_{3}(z)
  \end{align}
  with
  \begin{align}
    \label{eq:choicePL}
    \begin{split}
      H_{2}(z)
      := \exp\left(\beta_2 \log(a_2 z)\left(1+\frac{i}{2\gamma}\log(a_2z)\right)\right)
      = (a_2z)^{\beta_2(1+\frac{i}{2\gamma}\log(a_2z))}
    \end{split}
  \end{align}
  and
  \begin{align}
    \label{eq:choicePL2}
    H_{3}(z)
    := \exp\left(-\beta_3 \log(a_3 z)\left(1+\frac{i}{2\gamma}\log(a_3z)\right)\right)
    = (a_3z)^{-\beta_3(1+\frac{i}{2\gamma}\log(a_3 z))}\,.
  \end{align}
  Here $\log(|z|\me{i\theta}):=\log(|z|)+i\theta$ with $|\theta|<\pi$
  is the principal branch of the logarithm.
  By the assumptions on $\|F(z)\|$ for $\theta=0$ and $\theta=\gamma$,
  and the bounds
  \begin{align*}
    |H_{2}(|z|)| \leq (a_2|z|)^{\beta_2}
  \end{align*}
  and
  \begin{align*}
    |H_{2}(|z|\me{i\gamma})|
    & = \left|\exp\left(\frac{i\beta_2}{2}\left((\log|a_2z|)^2/\gamma + \gamma\right)\right)\right| = 1\,,
  \end{align*}
  and
  \begin{align*}
    |H_{3}(|z|)|  \leq (a_3|z|)^{-\beta_3}
    \qquad \text{and} \qquad
    |H_{3}(|z|\me{i\gamma})| =1\,,
  \end{align*}
  respectively, we have
  \begin{align*}
    \begin{split}
      \|G(|z|)\|
      & \leq |z|^{-\beta_1} \cdot a_1 |z|^{\beta_1}\left(a_2|z|\right)^{-\beta_2} (a_2|z|)^{\beta_2} \cdot (a_3|z|)^{\beta_3} (a_3|z|)^{-\beta_3} \leq a_1
    \end{split}
  \end{align*}
  and
  \begin{align*}
    \begin{split}
      \|G(|z|\me{i\gamma})\|
      & \leq |z|^{-\beta_1} \cdot a_1 |z|^{\beta_1}(\cos\gamma)^{-\beta_1}
      = a_1(\cos\gamma)^{-\beta_1}\,.
    \end{split}
  \end{align*}
  Combining the above two formulas using the three lines lemma shows
  \begin{align*}
    \|G(|z|\me{i\theta})\|
    \leq a_1(\cos\gamma)^{-\beta_1\theta/\gamma}
    \leq a_1(\cos\gamma)^{-\beta_1} 
  \end{align*}
  for all $0\leq\theta\leq\gamma$ and $|z|>0$.
  Plugging this estimate and the identities
  \begin{align*}
    \begin{split}
      |H_{2}(z^{-1})^{-1}|
      & = |(a_2z^{-1})^{-\beta_2(1+\frac{i}{2\gamma}\log(a_2z^{-1}))}|
      = (a_2|z|^{-1})^{-\beta_2(1+\theta/\gamma)}\\
      |H_{3}(z^{-1})^{-1}|
      & = (a_3|z|^{-1})^{\beta_3(1+\theta/\gamma)}
    \end{split}
  \end{align*}
  into the expression for $F(z)$ in \eqref{eq:defg}, i.e.,
  \begin{align*}
    F(|z|\me{i\theta}) = z^{-\beta_1}G(z^{-1}) H_{2}(z^{-1})^{-1}\cdot H_{3}(z^{-1})^{-1}\,,
  \end{align*}
  implies for $-\gamma\leq\theta<0$ and $|z|>0$,
  \begin{align}
    \label{eq:boundf}
    \begin{split}
      \|F(|z|\me{i\theta})\|
      & \leq a_1 |z|^{-\beta_1}(\cos\gamma)^{-\beta_1} (a_2|z|^{-1})^{-\beta_2(1+\theta/\gamma)} \cdot (a_3|z|^{-1})^{\beta_3(1+\theta/\gamma)}\,.
    \end{split}
  \end{align}
  By a reflection along the real axis we conclude the corresponding bound
  with $\theta$ replaced by $-\theta\in[-\gamma,0)$.
  Choosing
  $\gamma=\gamma(\epsilon,\theta)=\epsilon|\theta|+(1-\epsilon)\pi/2$
  for any $0<\epsilon<1$ (which ensures $|\theta|<\gamma<\pi/2$)
  lets us estimate
  $(\cos\gamma)^{-\beta_1} \leq \epsilon^{-\beta_1} (\cos\theta)^{-\beta_1}$
  and conclude the proof of \eqref{eq:pl1}, upon taking the minimum between
  \eqref{eq:boundf} and \eqref{eq:assum1}.
\end{proof}

\begin{remarks}
  (1) For $\beta_3=0$ 
  one observes that the decay of $|F(z)|$ in the regime $a_2\gg|z|$
  becomes weaker as $|\theta|$ increases.

  (2) The choice $\gamma=\epsilon|\theta|+(1-\epsilon)\pi/2$ suppresses
  the effect of large $|\theta|$ for $\epsilon\ll1$, but does not prevent
  the vanishing of decay as $|\theta|\to\pi/2$.
\end{remarks}

\section{Application - Semigroups of fractional Schr\"odinger operators}
\label{s:application}

Let $d\in\N$, $\alpha>0$, $V\in L_\loc^1(\R^d:\R)$, and recall the
notation \eqref{eq:defhalpha} for the fractional Schr\"odinger operator
$H_\alpha=(-\Delta)^{\alpha/2}+V$ in $L^2(\R^d)$.
We assume that $V$ is such that the quadratic form of $H_\alpha$ is
non-negative on $C_c^\infty(\R^d)$ so that it gives rise to a
self-adjoint operator by Friedrichs' theorem. In particular,
$\me{-zH_\alpha}$ is a bounded, holomorphic semigroup on $L^2(\R^d)$
whenever $\re(z)>0$ \cite[p.~493]{Kato1976}.
In the following we derive certain weighted and averaged estimates for
the extension of the semigroup $\me{-tH_\alpha}$ to complex times
$z\in\C_+$ and distinguish between the cases where $V\geq0$ and where
$V$ has a non-vanishing negative part, respectively.

\subsection{Non-negative potentials}
\label{ss:posv}

If $\alpha\in(0,2)$, then
$\me{-t(-\Delta)^{\alpha/2}}(x)\sim_{d,\alpha}t^{-d/\alpha}(1+|x|/t^{1/\alpha})^{-d/\alpha}$,
i.e., the kernel is in particular positive \cite{BlumenthalGetoor1960}.
Therefore, Trotter's formula implies for $V\geq0$ the upper bound
\begin{align}
  \label{eq:heatkernelv}
  \me{-tH_\alpha}(x,y)
  \lesssim_{d,\alpha} t^{-\frac{d}{\alpha}}\left(1+\frac{|x-y|}{t^{1/\alpha}}\right)^{-d-\alpha}\,.
\end{align}
In turn, \eqref{eq:heatkernelv} implies the following weighted
$L^2\to L^\infty$ estimate for any $r,t>0$,
\begin{align}
  \label{eq:keysmallt}
  \sup_{y\in\R^d}\int\limits_{\R^d\setminus B_y(r)}\left|\exp(-tH_\alpha)(x,y)\right|^2\,dx
  \lesssim t^{-\frac d\alpha}\left(1+\frac{r}{t^{1/\alpha}}\right)^{-d-2\alpha}\,.
\end{align}
The uniform complex-time heat kernel estimates \eqref{eq:zhaozheng}
by Zhao and Zheng yield the following extension to $z\in\C_+$, i.e.,
\begin{align}
  \label{eq:keysmallz}
  \begin{split}
    & \sup_{y\in\R^d}\int\limits_{\R^d\setminus B_y(r)}\left|\exp(-zH_\alpha)(x,y)\right|^2\,dx\\
    & \quad \lesssim \me{2\mu_{\epsilon,V}|z|} \cdot (\cos\theta)^{-2d(\frac1\alpha-\frac12)\one_{\{\alpha<1\}} - 2(\frac{d}{2}+\alpha-1)\one_{\{\alpha\geq1\}}} |z|^{-\frac d\alpha}\left(1+\frac{r}{|z|^{1/\alpha}}\right)^{-d-2\alpha}\,.
  \end{split}
\end{align}

In some applications one is merely in possession of averaged
and possibly weighted analogs of \eqref{eq:heatkernelv},
such as \eqref{eq:keysmallt}.
This is typically the case when $V$ has a singular negative part,
which will be discussed in more detail in Section~\ref{ss:negv}.
The following corollary illustrates how the Phragm\'en--Lindel\"of
principle can be used to extend weighted $L^2\to L^\infty$ estimates
like \eqref{eq:keysmallt} to complex times.

\begin{corollary}
  \label{PLapplied}
  Let $\alpha\in(0,2)$ and $V\geq0$.
  Let further $z=|z|\me{i\theta}$ with $|\theta|\in[0,\pi/2)$,
  $r>0$, $\epsilon\in(0,1)$, and
  \begin{align}
    \beta_{d,\alpha,\epsilon}(\theta)
    := (d+2\alpha) \left(1-\frac{|\theta|}{\epsilon|\theta|+(1-\epsilon)\pi/2}\right) \geq0\,.
  \end{align}
  Then
  \begin{align}
    \label{eq:PLapplied0}
    \begin{split}
      \sup_{y\in\R^d}\int\limits_{\R^d\setminus B_y(r)}\left|\exp(-zH_\alpha)(x,y)\right|^2\,dx
      \lesssim_{d,\alpha,\epsilon} (|z|\cos\theta)^{-\frac{d}{\alpha}}\left(1+\frac{r}{|z|^{1/\alpha}}\right)^{-\beta_{d,\alpha,\epsilon}(\theta)}\,.
    \end{split}
  \end{align}
\end{corollary}

Although the singularity of $(\cos\theta)^{-d/\alpha}$ in
\eqref{eq:PLapplied0} is less severe than in \eqref{eq:keysmallz},
the decay in the region $r\gg|z|^{1/\alpha}$ becomes weaker as
$|\theta|$ increases.

\begin{proof}
  For $y\in\R^d$ define the analytic function $F_y:\C_+\to\C$ by
  \begin{align*}
    F_y(z) = \left(\int_{\R^d}\me{-zH_\alpha}(x,y)f(x)\,dx\right)^2
  \end{align*}
  for any normalized $f\in L^2(\R^d)$ with
  $\supp f\subseteq \R^d\setminus B_y(r)$.
  For $|\theta|<\frac\pi2$ and $z=|z|\me{i\theta}$, we have, using
  the unitarity of $\me{-itH_\alpha}$ and
  \eqref{eq:heatkernelv} with $t$ replaced by $|z|\cos\theta$,
  \begin{align*}
    \begin{split}
      \sup_{y\in\R^d}|F_y(z)|
      & \leq \|\me{-zH_\alpha}\|_{2\to\infty}^2
      \leq \|\me{-|z|\cos\theta H_\alpha}\|_{2\to\infty}^2\|\me{-i|z|\sin\theta H_\alpha}\|_{2\to2}^2\\
      & \leq C_{1,d,\alpha} (|z|\cos\theta)^{-\frac{d}{\alpha}}\,.
    \end{split}
  \end{align*}
  On the other hand we have for $\theta=0$,
  \begin{align*}
    \sup_{y\in\R^d}|F_y(|z|)|
    & \leq \int_{\R^d\setminus B_y(r)}\left|\me{-|z|H_\alpha}(x,y)\right|^2\,dx
      \leq C_{2,d,\alpha} |z|^{-d/\alpha}\left(1+\frac{r}{|z|^{1/\alpha}}\right)^{-d-2\alpha}\\
    & \leq C_{2,d,\alpha} |z|^{-d/\alpha}\left(\frac{r^\alpha}{|z|}\right)^{-\frac{d}{\alpha}-2}
  \end{align*}
  by \eqref{eq:keysmallt}.
  Thus, we may apply Theorem~\ref{PL} with $X=\C$, 
  $a_1=\max\{C_{1,d,\alpha},C_{2,d,\alpha}\}$,
  $a_2=r^\alpha$,
  $\beta_1=d/\alpha$, 
  $\beta_2=(d+2\alpha)/\alpha$,
  and $\beta_3=0$, and obtain
  \begin{align*}
    |F_y(z)|
    \lesssim_{d,\alpha,\epsilon} (|z|\cos\theta)^{-\frac d\alpha}\left[1\wedge \left(\frac{r^\alpha}{|z|}\right)^{-\frac{d+2\alpha}{\alpha}\left(1-\frac{|\theta|}{\epsilon|\theta|+(1-\epsilon)\pi/2}\right)}\right]\,,
  \end{align*}
  which shows \eqref{eq:PLapplied0}.
\end{proof}

Similarly, one can use Theorem~\ref{PL} to derive pointwise estimates.

\begin{corollary}
  \label{PLapplied2}
  Let $\alpha\in(0,2)$ and $V\geq0$.
  Let further $z=|z|\me{i\theta}$ with $|\theta|\in[0,\pi/2)$,
  $x,y\in\R^d$, $\epsilon\in(0,1)$, and
  \begin{align}
    \label{eq:PLapplied2defbeta}
    \beta_{d,\alpha,\epsilon}(\theta)
    := (d+\alpha) \left(1-\frac{|\theta|}{\epsilon|\theta|+(1-\epsilon)\pi/2}\right) \geq0\,.
  \end{align}
  Then
  \begin{align}
    \label{eq:PLapplied2}
    \begin{split}
      \left|\me{-zH_\alpha}(x,y)\right|
      \lesssim_{d,\alpha,\epsilon} (|z|\cos\theta)^{-\frac{d}{\alpha}}\left(1+\frac{|x-y|}{|z|^{1/\alpha}}\right)^{-\beta_{d,\alpha,\epsilon}(\theta)}\,.
    \end{split}
  \end{align}
\end{corollary}

\begin{proof}
  For $z=|z|\me{i\theta}$ with $|\theta|\in[0,\pi/2)$
  we use $\|\me{-i\im(z)H_\alpha}\|_{2\to2}=1$ and estimate
  \begin{align*}
    \|\me{-zH_\alpha}\|_{1\to\infty}
    \leq \|\me{-\re(z)H_\alpha/2}\|_{1\to2}\|\me{-\re(z)H_\alpha/2}\|_{2\to\infty}
    \leq C_{1,d,\alpha} (|z|\cos\theta)^{-\frac d\alpha}\,,
  \end{align*}
  which shows
  \begin{align*}
    |\me{-zH_\alpha}(x,y)|
    \leq C_{1,d,\alpha}(|z|\cos\theta)^{-\frac d\alpha}\,,
    \quad x,y\in\R^d\,.
  \end{align*}
  On the other hand, \eqref{eq:heatkernelv} also implies
  \begin{align*}
    |\me{-|z|H_\alpha}(x,y)|
    \leq C_{2,d,\alpha}|z|^{-\frac d\alpha}\left(\frac{|x-y|^\alpha}{|z|}\right)^{-\frac{d+\alpha}{\alpha}}\,,
    \quad x,y\in\R^d\,.
  \end{align*}
  Thus, by Theorem~\ref{PL} with $X=\C$, 
  $a_1=\max\{C_{1,d,\alpha},C_{2,d,\alpha}\}$,
  $a_2=|x-y|^\alpha$,
  $\beta_1=d/\alpha$, 
  $\beta_2=(d+\alpha)/\alpha$,
  and $\beta_3=0$, we obtain
  \begin{align*}
    |\me{-zH_\alpha}(x,y)|
    \lesssim_{d,\alpha,\epsilon} (|z|\cos\theta)^{-\frac d\alpha}\left[1 \wedge \left(\frac{|x-y|^\alpha}{|z|}\right)^{-\frac{d+\alpha}{\alpha}\left(1-\frac{|\theta|}{\epsilon|\theta|+(1-\epsilon)\pi/2}\right)}\right]
  \end{align*}
  for all $x,y\in\R^d$. This shows \eqref{eq:PLapplied2}.
\end{proof}

\subsection{Potentials with a negative part}
\label{ss:negv}

We now consider the situation where the semigroup $\me{-tH_\alpha}$
is not necessarily $L^2\to L^\infty$ bounded anymore. This typically
occurs when $V$ has a non-vanishing, singular negative part
\cite{Liskevichetal2002,MilmanSemenov2004}.
In these cases $\me{-tH_\alpha}$ may nevertheless satisfy certain
weighted $L^p\to L^q$ estimates with $1<p\leq q<\infty$.
To introduce the estimates that we discuss here, we denote by
\begin{align*}
  A_2(x,r,k) := B_x(2^{k}r)\setminus B_x(2^{k-1} r)
  \quad \text{with} \quad
  A_2(x,r,0) := B_x(r)
  \quad \text{and} \quad
  k\in\N_0
\end{align*}
dyadic annuli around $x\in\R^d$.

\begin{definition}
  \label{gge}
  Let $r>0$, $(T_r)_{r>0}$ be a family of linear bounded operators
  on $L^2(\R^d)$,
  $1\leq p \leq q\leq\infty$,
  $\beta,\sigma>0$,
  and $g:\R_+\to\R_+$ satisfy
  $g(\lambda)\sim_\beta(1+\lambda)^{-\beta}$.
  Then $T_r$ is said to satisfy the \emph{dyadic $(p,q,\sigma)$
    Davies--Gaffney estimate} if there is a finite constant
  $C_{\DG}=C_{\DG}(d,p,q,\beta,\sigma)>0$ such that
  \begin{align}
    \label{eq:gge3}
    \|\one_{B_x(r)}T_r\one_{A_2(x,r,k)}\|_{p\to q}
    \leq C_{\DG}\, r^{-d\left(\frac1p-\frac1q\right)} g(2^k)2^{\frac{kd}{\sigma}}\,,
    \quad x\in\R^d,k\in\N_0\,,
  \end{align}
  and if $\beta>d(1/\sigma+1/q)$.
\end{definition}

This definition is inspired by Davies \cite{Davies1995} and the
notion of generalized gaussian estimates introduced by
Blunck and Kunstmann, cf.~\cite[p.~920]{BlunckKunstmann2003}.
Heuristically, the projection onto $B_x(r)$ captures the singularity
of $T_r$ at $x$, whereas the projection onto $A_2(x,r,k)$ controls its
decay at distance $2^kr$.
We use dyadic annuli instead of annuli with constant thickness, since
they allow to exploit decay more effectively.
Similar estimates were, e.g., used by Schreieck and Voigt
\cite{SchreieckVoigt1994} on $L^p$ independence of the spectrum of
Schr\"odinger operators with form small potentials, and later
systematically studied and exploited in a series of works by Blunck and Kunstmann
\cite{BlunckKunstmann2002,BlunckKunstmann2003,Blunck2003,BlunckKunstmann2004,BlunckKunstmann2005,Blunck2007}
on spectral multiplier theorems for operators whose semigroups need not
have a bounded kernel but obey the mentioned generalized gaussian
estimates.

Estimate \eqref{eq:gge3} with $p=1$, $q=\infty$, and the relaxed
assumption $\beta\geq d/\sigma$ is equivalent to a pointwise estimate
for the kernel of $T_r$.

\begin{proposition}
  Suppose $r,\beta,\sigma>0$, $\beta\geq d/\sigma$, and $(T_r)_{r>0}$
  is a family of linear operators in $\mathcal{B}(L^1\to L^\infty)$
  with integral kernel $T_r(x,y)$. Then the following statements are
  equivalent.

  (1) $T_r$ satisfies \eqref{eq:gge3} with $p=1$, $q=\infty$, and
  $C_{\DG}$ replaced by $c_{\beta,d,\sigma} C_{\DG}$ for some
  $c_{\beta,d,\sigma}>0$.

  (2) One has
  $|T_r(x,y)|\lesssim_{\beta,d,\sigma} C_{\DG}\, r^{-d}(1+|x-y|/r)^{-(\beta-d/\sigma)}$
  for all $x,y\in\R^d$.
\end{proposition}

\begin{proof}
  (1) $\Rightarrow$ (2): In this case \eqref{eq:gge3} asserts
  \begin{align*}
    \sup_{y,z\in\R^d}\one_{B_x(r)}(y)|T_r(y,z)|\one_{A_2(x,r,k)}(z)
    & \leq c_{\beta,d,\sigma} C_{\DG}\, r^{-d}g(2^k)2^{kd/\sigma}\\
    & \lesssim_{\beta,d,\sigma} C_{\DG}\, r^{-d}2^{-k(\beta-d/\sigma)}
  \end{align*}
  for all $x\in\R^d$ and $k\in\N_0$. Choosing $y=x$ and $k\in\N_0$ such that
  $2^k=\max\{1,|x-z|/r\}$ yields
  $|T_r(x,z)| \lesssim_{\beta,d,\sigma} C_{\DG}\, r^{-d}(1+|x-z|/r)^{-(\beta-d/\sigma)}$
  for all $x,z\in\R^d$.

  (2) $\Rightarrow$ (1): We estimate
  \begin{align*}
    \|\one_{B_x(r)}T_r\one_{A_2(x,r,k)}\|_{1\to \infty}
    & = \sup_{y,z\in\R^d}\one_{B_x(r)}(y)|T_r(y,z)|\one_{A_2(x,r,k)}(z)\\
    & \leq \sup_{y\in B_x(r)}\sup_{z\in \R^d\setminus B_x(2^{k-1}r)} |T_r(y,z)|\\
    & \lesssim_{\beta,d,\sigma} C_{\DG}\, r^{-d} (1+2^k \theta(k-2))^{-(\beta-d/\sigma)}\,,
  \end{align*}
  which concludes the proof.
\end{proof}

Further consequences of Definition \ref{gge} will be discussed in 
Section~\ref{ss:consequencesgge}.
They play an important role in the subsequent analysis.

\medskip
The example of the semigroup of $H_\alpha$ with the Hardy potential
$V=a|x|^{-\alpha}$  illustrates that \eqref{eq:gge3} is a reasonable
assumption.
Note that $|x|^{-\alpha}\notin K_\alpha(\R^d)$. The resulting operator
\begin{align}
  \label{eq:genhardy}
  \cl_{a,\alpha} := (-\Delta)^{\alpha/2} + \frac{a}{|x|^\alpha}
  \quad \text{in}\ L^2(\R^d)
\end{align}
for $\alpha\in(0,2\wedge d)$ and $a\geq a_*\equiv a_*(\alpha,d)>-\infty$
is sometimes called generalized or fractional Hardy operator.
By the sharp Hardy--Kato--Herbst inequality, $\cl_{a,\alpha}$ is bounded
from below and non-negative in the sense of quadratic forms if and only if
$a\geq a_*$. We refer to Kato
\cite[Chapter 5, Equation (5.33)]{Kato1976} (for $\alpha=1$ and $d=3$)
and Herbst \cite[Equation (2.6)]{Herbst1977} and also
\cite{Franketal2008H,FrankSeiringer2008,Kovalenkoetal1981,Yafaev1999}
for proofs of this fact and the explicit expression of $a_*$.

For $a\in[a_*,0]$ Bogdan et al.~\cite{Bogdanetal2019}
proved pointwise heat kernel bounds, namely
\begin{align}
  \label{eq:genhardyheat}
  \me{-t\cl_{a,\alpha}}(x,y)
  \sim_{d,\alpha,a} \big(1\vee\frac{t^{\frac1\alpha}}{|x|}\big)^\delta \big(1\vee\frac{t^{\frac1\alpha}}{|y|}\big)^\delta \frac{t}{(t^{\frac1\alpha}+|x-y|)^{d+\alpha}}\,,
  \quad t\!>\!0,\,x,y\in\R^d\setminus\{0\}\,,
\end{align}
where $\delta=\delta(a,d,\alpha)\in[0,(d-\alpha)/2]$ satisfies
$\delta(0,d,\alpha)=0$, $\delta(a_*,d,\alpha)=(d-\alpha)/2$,
and increases monotonously as $a$ decreases. An explicit formula
for $\delta(a,d,\alpha)$ is, e.g., contained in \cite{Bogdanetal2019}
or Frank et al.~\cite{Franketal2008H}.
In particular, $\me{-t\cl_{a,\alpha}}$ is $L^p\to L^p$ bounded for $a<0$
if and only if $p\in(d/(d-\delta),d/\delta)$. This follows from
\eqref{eq:genhardyheat}, duality, and
$(\me{-\cl_{a,\alpha}}\one_{B_0(1)})(x)\gtrsim|x|^{-\delta}$.
Moreover, $\me{-t\cl_{a,\alpha}}$ satisfies the $L^2\to L^\infty$
estimates in \eqref{eq:keysmallt} if and only if $a\geq0$.

\begin{example}
  \label{application}
  Let $\alpha\in(0,2\wedge d)$, $t>0$, and $r_t:=t^{1/\alpha}$.

  (1) If $a\in[a_*,0)$ and $p\in(d/(d-\delta),2]$, then
  $T_{r_t}:=\me{-t\cl_{a,\alpha}}$ satisfies the dyadic $(p,p',p')$,
  $(p,2,p')$, and $(2,p',2)$ Davies--Gaffney estimate \eqref{eq:gge3}
  with $g(\lambda)\sim_{d,\alpha}(1+\lambda)^{-d-\alpha}$.

  (2) If $a\geq0$ and $1\leq p\leq2\leq q\leq\infty$,
  then $T_{r_t}:=\me{-t\cl_{a,\alpha}}$ satisfies the dyadic $(p,q,p')$
  Davies--Gaffney estimate \eqref{eq:gge3} with
  $g(\lambda)\sim_{d,\alpha}(1+\lambda)^{-d-\alpha}$.
\end{example}

\begin{proof}  
  (1) By H\"older's inequality and \eqref{eq:genhardyheat}, one
  obtains for $x\in\R^d$ and $k\in\N\setminus\{1\}$
  \begin{align*}
    & \sup_{f\in L^p:\,\|f\|_p=1}\|\one_{B_{x}(r_t)}\me{-t\cl_{a,\alpha}}\one_{A_2(x,r_t,k)}f\|_{p'}^{p'}\\
    & \quad \leq \int\limits_{\frac{|z-x|}{r_t}\leq1} dz\, \left(1\vee\frac{r_t}{|z|}\right)^{\delta p'} \int\limits_{\frac{|y-x|}{r_t}\in[2^{k-1},2^k]} dy\, \left(1\vee\frac{r_t}{|y|}\right)^{\delta p'}\frac{r_t^{\alpha p'}}{(r_t+|z-y|)^{p'(d+\alpha)}}\\
    & \quad = r_t^{-d(p'-2)}(1+2^k)^{-(d+\alpha)p'} \cdot 2^{kd}
  \end{align*}
  where we used $|z-y| \geq |y-x|-r_t\geq(2^{k-1}-1)r_t$.
  For $k\in\{0,1\}$ it suffices to integrate $y$ over $|y-x|\leq2r_t$,
  in which case the left side is bounded by a constant times
  $r_t^{-d(p'-2)}\lesssim r_t^{-d(p'-2)}(1+2^k)^{-(d+\alpha)p'} \cdot 2^{kd}$
  as well. Taking the $p'$-th root yields the dyadic $(p,p',p')$
  Davies--Gaffney estimate \eqref{eq:gge3}.
  The $(p,2,p')$ and $(2,p',2)$ Davies--Gaffney estimates follow
  similarly. For $x\in\R^d$ and $k\in\N_0$, one has
  \begin{align*}
    & \sup_{f\in L^2:\,\|f\|_2=1}\|\one_{B_{x}(r_t)}\me{-t\cl_{a,\alpha}}\one_{A_2(x,r_t,k)}f\|_{p'}^{p'}\\
    & \quad \leq \int\limits_{\frac{|z-x|}{r_t}\leq1} dz\, \left(1\vee\frac{r_t}{|z|}\right)^{\delta p'} \left[\int\limits_{\frac{|y-x|}{r_t}\in[2^{k-1},2^k]} \! dy\, \left(1\vee\frac{r_t}{|y|}\right)^{2\delta} \frac{r_t^{2\alpha}}{(r_t+|z-y|)^{2(d+\alpha)}}\right]^{\frac{p'}{2}}\\
    & \quad \lesssim r_t^{-d(p'-2)/2}(1+2^k)^{-(d+\alpha)p'} \cdot 2^{kdp'/2}
  \end{align*}
  which shows the $(2,p',2)$ estimate. The $(p,2,p')$ estimate is
  shown analogously.

  (2) By \eqref{eq:heatkernelv} (Trotter's formula) and a similar
  computation, one obtains
  \begin{align*}
    \begin{split}
      & \|\one_{B_{x}(r_t)}\me{-t\cl_{a,\alpha}}\one_{A_2(x,r_t,k)}\|_{p\to q}
      \lesssim r_t^{-d(\frac1p-\frac1q)}\left(1+2^k\right)^{-(d+\alpha)}\cdot 2^{\frac{kd}{p'}}
    \end{split}
  \end{align*}
  for all $x\in\R^d$ and $k\in\N_0$.
\end{proof}

\begin{remark}
  In fact, \eqref{eq:genhardyheat} allows to derive pointwise
  estimates for $\me{-z\cl_{a,\alpha}}$ using Davies' method,
  cf.~\cite[Theorem~A]{MilmanSemenov2004} and Bui and D'Ancona
  \cite[Proposition~3.4]{BuiDAncona2021} when $|\arg(z)|<\pi/4$.
\end{remark}

We now use Theorem~\ref{PL} to extend dyadic Davies--Gaffney
estimates \eqref{eq:gge3} for $\me{-tH_\alpha}$ and $t>0$ to
complex times $z\in\C_+$.

\begin{theorem}
  \label{plgge}
  Let $\alpha>0$, $\beta>0$, $1\leq p\leq2\leq q\leq\infty$,
  and $\sigma,t>0$.
  Suppose $\me{-tH_\alpha}$ satisfies
  the dyadic $(p,q,\sigma)$ Davies--Gaffney estimate \eqref{eq:gge3}
  (Definition \ref{gge}) with $r\equiv r_t:=t^{1/\alpha}$, i.e., there
  is a constant $C_{\DG}>0$ such that
  \begin{align}
    \label{eq:assumpl}
    \|\one_{B_x(r_t)}\me{-tH_\alpha}\one_{A_2(x,r_t,k)}\|_{p\to q}
    \lesssim_\beta C_{\DG}\, r_t^{-d\left(\frac{1}{p}-\frac{1}{q}\right)} 2^{-k(\beta-\frac{d}{\sigma})}\,,
    \quad x\in\R^d\,,k\in\N_0\,.
  \end{align}
  In case $p\in[1,2)$ and $q\in(2,\infty]$, and $q\neq p'$,
  assume additionally the bounds
  $\max\{\|\me{-tH_\alpha}\|_{p\to p},\|\me{-tH_\alpha}\|_{q\to q}\}\lesssim_{d,p,q,\alpha,\beta,\sigma}1$.
  Then, for $z=|z|\me{i\theta}\in\C_+$, $\zeta\geq0$,
  $r_z:=|z|^{1/\alpha}(\cos\theta)^{-\zeta}$, $\epsilon\in(0,1)$, and
  \begin{align}
    \label{eq:deftildebeta}
    \begin{split}
      \tilde\beta & \equiv \tilde\beta_{d,\beta,\sigma,\epsilon}(\theta)
      := \left(\beta-\frac{d}{\sigma}\right) \cdot \left(1-\frac{|\theta|}{\epsilon|\theta|+(1-\epsilon)\pi/2}\right)
      \geq 0\,,
    \end{split}
  \end{align}
  one has 
  \begin{align}
    \label{eq:plgge}
    \begin{split}
      & \|\one_{B_x(r_z)}\me{-zH_\alpha}\one_{A_2(x,r_z,k)}\|_{p\to q}\\
      & \quad \lesssim_{d,\alpha,\beta,\sigma,p,q,\epsilon} C_{\DG}\, (|z|\cos\theta)^{-\frac{d}{\alpha}\left(\frac{1}{p}-\frac{1}{q}\right)} \cdot (\cos\theta)^{-\frac{d\zeta}{q}} \cdot 2^{-k\tilde\beta} \\
      & \quad = C_{\DG}\, r_z^{-d\left(\frac{1}{p}-\frac{1}{q}\right)}(\cos\theta)^{-d\left(\zeta+\frac{1}{\alpha}\right)\left(\frac{1}{p}-\frac{1}{q}\right) -\frac{d\zeta}{q}} \cdot 2^{-k\tilde\beta}
    \end{split}
  \end{align}
  for all $x\in\R^d$ and $k\in\N_0$.
  Moreover,
  \begin{align}
    \label{eq:plggeintegrated}
    \|\me{-zH_\alpha}\|_{p\to q}
    \lesssim_{d,\alpha,p,q,\beta,\sigma} C_{\DG}(|z|\cos\theta)^{-\frac{d}{\alpha}\left(\frac{1}{p}-\frac{1}{q}\right)}\,.
  \end{align}
\end{theorem}

\begin{remarks}
  (1) Example \ref{application} and \eqref{eq:zhaozheng} indicate
  that one will have $\beta=d+\alpha$ in many scenarios.
  Nevertheless, we prefer to keep $\beta$ as a free parameter here
  and in the following to illustrate that certain estimates for
  $\me{-tH_\alpha}$ can be extended to complex times under less
  severe decay conditions.
  
  (2) We make some remarks on the choices $r_t=t^{\frac1\alpha}$ and
  $r_z=|z|^{\frac1\alpha}(\cos\theta)^{-\zeta}$. The power $\alpha^{-1}$
  reflects the scaling relation between time and space in
  $\me{-zH_\alpha}$ and
  is dictated by the order of the principal symbol of 
  $(-\Delta)^{\frac{\alpha}{2}}\!+\!\min\{V,0\}$.
  For $t\!>\!0$ this is seen in \eqref{eq:heatkernel} for $V\!=\!0$,
  in \eqref{eq:heatkernelv} for $\alpha\in(0,2)$ and $V\!\geq\!0$,
  in \eqref{eq:genhardyheat} for $V=a|x|^{-\alpha}$ with $a\geq a_*$,
  and in Huang et al.~\cite[Theorem~1.3]{Huangetal2018} when
  $V\in K_\alpha(\R^d)$ is a perturbation.
  For complex times the power is expected to be $\alpha^{-1}$, too.
  This is confirmed by Corollaries \ref{PLapplied} and
  \ref{PLapplied2} for $V\geq0$, \eqref{eq:zhaozheng} for
  $V\in K_\alpha(\R^d)$, and estimates for $\me{-zH_\alpha}$ when
  $\alpha\in2\N$, see \eqref{eq:plgaussian} below.
  On the other hand, a natural choice for $\zeta$ is not obvious
  due to the complicated relation between $\theta$ and $|z|$ in
  \eqref{eq:plgge}. For $\alpha\in 2\N$, $V\geq0$, $x,y\in\R^d$,
  $r,s>0$, and $|x-y|>r+s$, Davies \cite[Theorem~10]{Davies1995}
  showed
  \begin{align}
    \label{eq:plgaussian}
    \|\one_{B_x(r)}\me{-|z|\me{i\theta}H_\alpha}\one_{B_y(s)}\|_{2\to 2}
    \leq \exp\!\left(-c\left(\frac{d(B_x(r),B_y(s))}{|z|^{\frac{1}{\alpha}}(\cos\theta)^{-\frac{\alpha-1}{\alpha}}}\right)^{\frac{\alpha}{\alpha-1}}\right)
  \end{align}
  with similar estimates being available for more general (even singular)
  $V$, cf.~\cite[pp.~154-156]{Blunck2007}. This shows that
  $r_z=|z|^{1/\alpha}(\cos\theta)^{-\zeta}$ with
  $\zeta\equiv(\alpha-1)/\alpha>0$ is a natural choice when $\alpha\in2\N$.
  In fact, $\zeta>0$ is necessary for \eqref{eq:plgaussian} to hold on
  all of $\C_+$, which can be seen by fixing $x\neq y$ and letting $|\theta|\to\pi/2$.
  
  (3) Estimate \eqref{eq:plgaussian} and many related variants, such as
  \cite[Proposition~4.1]{Carronetal2002} by Carron et al.~and
  \cite[Theorem~2.1]{Blunck2007} by Blunck, were proved using
  Phragm\'en--Lindel\"of principles, see Davies \cite[Lemma~9]{Davies1995}
  for the original version.
  The presence of exponential bounds for $\alpha\in2\N$ and $t>0$
  is essential for the derivation of the clean dependence on
  \eqref{eq:plgaussian} on $\cos\theta$.  
\end{remarks}

The proof of Theorem~\ref{plgge} is inspired by Blunck
\cite[Theorem~2.1]{Blunck2007} and uses two consequences
(Propositions~\ref{hypercontractive} and \ref{eqgge}) of the dyadic
Davies--Gaffney estimates that are contained in Section~\ref{ss:consequencesgge}.

\begin{proof}[Proof of Theorem~\ref{plgge}]
  For $\mu=1/2-1/q$ and $\nu=1/p-1/2$ we have
  \begin{align}
    \label{eq:plggeaux1}
    r_t^{d\mu}\|\me{-tH_\alpha}\|_{2\to q}\lesssim C_{\DG}^{\mu/(\mu+\nu)}
    \quad \text{and} \quad
    r_t^{d\nu}\|\me{-tH_\alpha}\|_{p\to2} \lesssim C_{\DG}^{\nu/(\mu+\nu)}\,,
    \quad t>0
  \end{align}
  by Proposition~\ref{hypercontractive}, whenever
  $p=2$ and $q\in[2,\infty)$, or $p\in[1,2]$ and $q=2$, or $q=p'$.
  In all other cases \eqref{eq:plggeaux1} follows from Riesz--Thorin
  interpolation between \eqref{eq:hypercontractive} in
  Proposition~\ref{hypercontractive} and the $L^p\to L^p$ and $L^q\to L^q$
  boundedness of $\me{-tH_\alpha}$.
  Thus, we obtain for $z=|z|\me{i\theta}\in\C_+$ and
  $t\in(0,\re(z)/2)$,
  \begin{align}
    \label{eq:plbound1pre}
    \begin{split}
      \|\me{-zH_\alpha}\|_{p\to q}
      \leq \|\me{-tH_\alpha}\|_{2\to q}\, \|\me{-(z-2t)H_\alpha}\|_{2\to2}\, \|\me{-tH_\alpha}\|_{p\to2}
      \lesssim C_{\DG}\, r_t^{-d\left(\frac{1}{p}-\frac{1}{q}\right)}\,.
    \end{split}
  \end{align}
  Combining \eqref{eq:plbound1pre} with
  $\lim_{t\nearrow\re z/2} r_t^{-1} = 2^{1/\alpha}(|z|\cos\theta)^{-\frac1\alpha}$
  yields
  \begin{align*}
    \|\me{-zH_\alpha}\|_{p\to q}
    \lesssim C_{\DG}(|z|\cos\theta)^{-\frac{d}{\alpha}\left(\frac{1}{p}-\frac{1}{q}\right)}\,,
  \end{align*}
  which proves \eqref{eq:plggeintegrated}.
  Moreover, for any ball $B\equiv B_{x_0}(r_0)$ and its $k$-th dyadic
  annulus $A_2^{(k)}\equiv A_2(x_0,r_0,k)$ centered around $B$
  (with $k\in\N_0$ and the convention $A_2^{(0)}=B$), we obtain
  \begin{align}
    \label{eq:plbound1}
    \begin{split}
      \|\one_{B}\me{-zH_\alpha}\one_{A_2^{(k)}}\|_{p\to q}
      \!\leq\! c_{d,\alpha,\beta,p,q,\sigma} C_{\DG}(|z|\cos\theta)^{-\frac{d}{\alpha}\left(\frac{1}{p}-\frac{1}{q}\right)} \,.
    \end{split}
  \end{align}
  On the other hand, \eqref{eq:gge3eq} in Proposition~\ref{eqgge}
  implies for $\theta=0$, 
  \begin{align}
    \label{eq:plbound2}
    \begin{split}
      \|\one_{B}\me{-tH_\alpha}\one_{A_2^{(k)}}\|_{p\to q}
      & \leq c_{d,\alpha,\beta,p,q,\sigma} C_{\DG}\, t^{-\frac{d}{\alpha}\left(\frac{1}{p}-\frac{1}{q}\right)} \cdot \left(1 + \frac{|B|^{\frac\alpha d}}{t}\right)^{\frac{d}{\alpha q}} \\
      & \quad \times \left[\left(\frac{d(B,A_2^{(k)})}{r_t}\right)^{-\beta}\cdot \left(\frac{|A_2^{(k)}|}{r_t^d}\right)^{\frac1\sigma}\theta(k-2)+\theta(1-k)\right]\,.
    \end{split}
  \end{align}
  Define the analytic function $F:\C_+\to\mathcal{B}(L^p\to L^q)$ by
  \begin{align}
    \label{eq:deffnegv}
    \begin{split}
      F(z)
      & := (c_{d,\alpha,\beta,p,q,\sigma} C_{\DG})^{-1}\, \left(1+\frac{|B|^{\frac\alpha d}}{z}\right)^{-\frac{d}{\alpha q}} \one_{B}\me{-zH_\alpha}\one_{A_2^{(k)}}\,.
    \end{split}
  \end{align}
  By \eqref{eq:plbound1} and \eqref{eq:plbound2} we have (using
  $|1+|B|^{\alpha/d}/z|\geq1$ for $z\in\C_+$)
  \begin{align*}
    \|F(|z|\me{i\theta})\|_{p\to q}
    & \leq (|z|\cos\theta)^{-\frac{d}{\alpha}\left(\frac{1}{p}-\frac{1}{q}\right)}\,,
  \end{align*}
  \begin{align*}
    \begin{split}
      \| F(|z|)\|_{p\to q}
      & \leq |z|^{-\frac{d}{\alpha}\left(\frac1p-\frac1q\right)} \left[\left(\frac{d(B,A_2^{(k)})}{|z|^{1/\alpha}}\right)^{-\beta}\cdot\left(\frac{|A_2^{(k)}|^{\frac\alpha d}}{|z|}\right)^{\frac{d}{\alpha\sigma}}\theta(k-2)+\theta(1-k)\right]\,.
    \end{split}
  \end{align*}
  We now apply Theorem~\ref{PL} with
  $a_1=1$, $\beta_1=d(1/p-1/q)/\alpha$, and
  $a_2=d(B,A_2^{(k)})^\alpha$, $\beta_2=\beta/\alpha$,
  $a_3=|A_2^{(k)}|^{\frac{\alpha}{d}}$, $\beta_3=\frac{d}{\alpha\sigma}$ if $k\in\N\setminus\{1\}$,
  and $a_2=a_3=1$ and $\beta_2=\beta_3=0$ if $k\in\{0,1\}$.
  Abbreviating $\gamma_\epsilon=\epsilon|\theta|+(1-\epsilon)\pi/2$,
  we obtain
  \begin{align*}
    \begin{split}
      & \|F(|z|\me{i\theta})\|_{p\to q}
      \lesssim_{\epsilon} (|z|\cos\theta)^{-\frac{d}{\alpha}\left(\frac{1}{p}-\frac{1}{q}\right)}\\
      & \quad \times\! \left\{\theta(1-k) \!+\! \theta(k-2)\left[1\wedge \left(\frac{d(B,A_2^{(k)})^\alpha}{|z|}\right)^{-\frac{\beta}{\alpha}}\cdot\left(\frac{|A_2^{(k)}|^{\frac{\alpha}{d}}}{|z|}\right)^{\frac{d}{\alpha\sigma}}\right]^{1-\frac{|\theta|}{\gamma_\epsilon}}\right\}
    \end{split}
  \end{align*}
  for all $|z|>0$ and $|\theta|<\pi/2$.
  By the definition \eqref{eq:deffnegv} and 
  \begin{align*}
    \left|1+\frac{|B|^{\alpha/d}}{z}\right|^{\frac{d}{\alpha q}}
    \lesssim \left(1+\frac{|B|}{|z|^{d/\alpha}}\right)^{1/q}\,,
  \end{align*}
  this implies
  \begin{align*}
    & \|\one_{B}\me{-zH_\alpha}\one_{A_2^{(k)}}\|_{p\to q}
      \lesssim C_{\DG} (|z|\cos\theta)^{-\frac{d}{\alpha}(\frac{1}{p}-\frac{1}{q})} \cdot\left(1+\frac{|B|}{|z|^{\frac d\alpha}}\right)^{\frac1q}\\
    & \quad \times \left\{\theta(1-k) + \theta(k-2)\left[1\wedge \left(\frac{d(B,A_2^{(k)})^\alpha}{|z|}\right)^{-\frac{\beta}{\alpha}}\cdot\left(\frac{|A_2^{(k)}|^{\frac{\alpha}{d}}}{|z|}\right)^{\frac{d}{\alpha\sigma}}\right]^{1-\frac{|\theta|}{\gamma_\epsilon}}\right\}\,.
  \end{align*}
  Choosing $B\equiv B_x(r_z)$ and $A_2^{(k)}\equiv A_2(x,r_z,k)$,
  and recalling $r_z=|z|^{1/\alpha}(\cos\theta)^{-\zeta}$ and
  $r_{|z|}=|z|^{1/\alpha}$ yields
  \begin{align*}
    \begin{split}
      & \|\one_{B_x(r_z)}\me{-zH_\alpha}\one_{A_2(x,r_z,k)}\|_{p\to q}\\
      & \quad \lesssim C_{\DG}(|z|\cos\theta)^{-\frac{d}{\alpha}\left(\frac{1}{p}-\frac{1}{q}\right)} \cdot \left(1+\frac{r_z}{r_{|z|}}\right)^{\frac{d}{q}}\\
      & \qquad \times \left\{\theta(1-k) + \theta(k-2)\left(\frac{2^kr_z}{r_{|z|}}\right)^{-(\beta-\frac{d}{\sigma})\cdot(1-\frac{|\theta|}{\gamma_\epsilon})}\right\}\\
      & \quad \lesssim C_{\DG}(|z|\cos\theta)^{-\frac{d}{\alpha}\left(\frac{1}{p}-\frac{1}{q}\right)} \cdot (\cos\theta)^{-\frac{d\zeta}{q}} \cdot 2^{-k\left(\beta-\frac d\sigma\right)\left(1-\frac{|\theta|}{\epsilon|\theta|+(1-\epsilon)\pi/2}\right)}
    \end{split}
  \end{align*}
  for all $z\in\C_+$. Here we used $\beta\geq d/\sigma$ and
  $r_z\geq r_{|z|}$ since $\zeta\geq0$. This shows \eqref{eq:plgge}
  and concludes the proof of Theorem~\ref{plgge}.
\end{proof}

For $p\in[1,2]$ Theorem~\ref{plgge} can be used to extend dyadic
$(2,p',\sigma)$ and $(p,2,\tilde\sigma)$ Davies--Gaffney estimates for
$\me{-tH_\alpha}$ to complex times.
In Proposition~\ref{hypercontractivedual} we show that for $t>0$,
$\sigma=2$, and $\tilde\sigma=p'$ they can be inferred from $(p,p',p')$
estimates under a slightly stronger decay condition.

\begin{definition}
  \label{ggerestricted}
  Let $r>0$, $(T_r)_{r>0}$ be a family of linear bounded operators
  on $L^2(\R^d)$,
  $1\leq p \leq q\leq\infty$,
  $\beta,\sigma>0$,
  and $g:\R_+\to\R_+$ satisfy
  $g(\lambda)\sim_\beta(1+\lambda)^{-\beta}$.
  Then $T_r$ is said to satisfy the
  \begin{enumerate}
  \item \emph{restricted dyadic $(p,q,\sigma)$ Davies--Gaffney estimate}
    if there is a finite constant $C_{\DG}=C_{\DG}(d,p,q,\beta,\sigma)>0$
    such that \eqref{eq:gge3} holds,
    and if $\beta>d(1/p+1/\sigma)$.

  \item \emph{dual dyadic $(p,p',p')$ Davies--Gaffney estimate} if $p\in[1,2]$,
    if there is a finite constant $C_{\DG}=C_{\DG}(d,p,\beta,\sigma)>0$
    such that \eqref{eq:gge3} holds, and if $\beta>d(1/2+1/p')$.
  \end{enumerate}
\end{definition}

\begin{remarks}
  (1) If $p\in[1,2]$, then the restricted dyadic $(p,p',p')$ estimate
  implies the dual dyadic $(p,p',p')$ and the dyadic $(p,p',p')$
  Davies--Gaffney estimate.
  
  (2) In Proposition~\ref{hypercontractivedual} we show that
  dual dyadic $(p,p',p')$ estimates imply dyadic $(2,p',2)$ and
  $(p,2,p')$ estimates, whereas restricted dyadic $(p,p',p')$ estimates
  imply restricted dyadic $(2,p',2)$ and $(p,2,p')$ Davies--Gaffney
  estimates.

  (3) The notions of dyadic $(p,p,\sigma)$ and restricted dyadic
  $(p,p,\sigma)$ Davies--Gaffney estimates coincide.

  (4) The semigroup $\me{-t\cl_{a,\alpha}}$ in Example \ref{application}
  satisfies the restricted dyadic $(p,p',p')$ estimates whenever
  $a\in[-a_*,0)$ and $p\in(d/(d-\delta),2]$. Moreover, if $V\geq0$,
  then $\me{-tH_\alpha}$ satisfies the restricted dyadic $(p,p',p')$
  Davies--Gaffney estimates for all $p\in[1,2]$.
\end{remarks}

The following corollary shows that dual dyadic $(p,p',p')$
Davies--Gaffney estimates estimates can be used to derive
complex-time $(2,p',2)$ and $(p,2,p')$ estimates.

\begin{corollary}
  \label{plggecor}
  Let $\alpha>0$, $\beta>0$, $1\leq p\leq2$, and $t>0$.
  Suppose $\me{-tH_\alpha}$ satisfies the dual dyadic $(p,p',p')$
  Davies--Gaffney estimate (Definition \ref{ggerestricted})
  with $r\equiv r_t:=t^{1/\alpha}$ and $g(\lambda)\sim_\beta(1+\lambda)^{-\beta}$.
  Then, for $z=|z|\me{i\theta}\in\C_+$,
  $\zeta\geq0$, $r_z:=|z|^{1/\alpha}(\cos\theta)^{-\zeta}$,
  $\epsilon\in(0,1)$, and
  \begin{align}
    \label{eq:deftildebetaagain}
    \begin{split}
      \tilde\beta^{(1)} & \equiv \tilde\beta_{d,\beta,\epsilon}^{(1)}(\theta)
      := \left(\beta-\frac{d}{2}\right) \left(1-\frac{|\theta|}{\epsilon|\theta|+(1-\epsilon)\pi/2}\right) \geq0 \,, \\
      \tilde\beta^{(2)} & \equiv \tilde\beta_{d,\beta,p,\epsilon}^{(2)}(\theta)
      := \left(\beta-\frac{d}{p'}\right) \left(1-\frac{|\theta|}{\epsilon|\theta|+(1-\epsilon)\pi/2}\right) \geq0 \,, \\
    \end{split}
  \end{align}
  one has
  \begin{align}
    \label{eq:plggecor}
    \begin{split}
      & \|\one_{B_x(r_z)}\me{-zH_\alpha}\one_{A_2(x,r_z,k)}\|_{2\to p'}\\
      & \quad \lesssim_{d,\alpha,p,\beta,\epsilon} C_\DG\, (|z|\cos\theta)^{-\frac{d}{\alpha}\left(\frac{1}{2}-\frac{1}{p'}\right)} (\cos\theta)^{-\frac{d\zeta}{p'}} \cdot 2^{-k\tilde\beta^{(1)}} \\
      & \quad = C_\DG\, r_z^{-d\left(\frac{1}{2}-\frac{1}{p'}\right)}(\cos\theta)^{-d\left(\zeta+\frac{1}{\alpha}\right)\left(\frac{1}{2}-\frac{1}{p'}\right)-\frac{d\zeta}{p'}} \cdot 2^{-k\tilde\beta^{(1)}}
    \end{split}
  \end{align}
  and
  \begin{align}
    \label{eq:plggecor2}
    \begin{split}
      & \|\one_{B_x(r_z)}\me{-zH_\alpha}\one_{A_2(x,r_z,k)}\|_{p\to 2}\\
      & \quad \lesssim_{d,\alpha,p,\beta,\epsilon} C_\DG\, (|z|\cos\theta)^{-\frac{d}{\alpha}\left(\frac{1}{2}-\frac{1}{p'}\right)} (\cos\theta)^{-\frac{d\zeta}{2}} \cdot 2^{-k\tilde\beta^{(2)}} \\
      & \quad = C_\DG\, r_z^{-d\left(\frac{1}{2}-\frac{1}{p'}\right)}(\cos\theta)^{-d\left(\zeta+\frac{1}{\alpha}\right)\left(\frac{1}{2}-\frac{1}{p'}\right)-\frac{d\zeta}{2}} \cdot 2^{-k\tilde\beta^{(2)}}
    \end{split}
  \end{align}
  for all $x\in\R^d$ and $k\in\N_0$.
  Moreover,
  \begin{align}
    \label{eq:plggecor3}
    \|\me{-zH_\alpha}\|_{2\to p'} = \|\me{-\overline zH_\alpha}\|_{p\to2}
    \lesssim_{d,\alpha,p,\beta} C_{\DG}(|z|\cos\theta)^{-\frac{d}{\alpha}\left(\frac{1}{p}-\frac{1}{2}\right)}\,.
  \end{align}
\end{corollary}

\begin{proof}
  By Proposition~\ref{hypercontractivedual} the dual dyadic $(p,p',p')$
  Davies--Gaffney estimate \eqref{eq:gge3} implies the dyadic $(2,p',2)$
  and $(p,2,p')$ Davies--Gaffney estimates.
  Thus, assumption \eqref{eq:assumpl} in Theorem~\ref{plgge} with
  $(p,q,\sigma)=(2,p',2)$ or $(p,q,\sigma)=(p,2,p')$ there is
  satisfied.
  The proof is concluded by an application of Theorem~\ref{plgge}.
\end{proof}

While the $L^p\to L^q$ estimates \eqref{eq:plggeintegrated} and
$L^2\to L^{p'}$ and $L^p\to L^2$ estimates \eqref{eq:plggecor3}
could be proved rather directly, one could have also obtained them
by combining Proposition~\ref{hypercontractive} with
\eqref{eq:plgge}, \eqref{eq:plggecor}, and \eqref{eq:plggecor2}.
However, this argument requires a smallness assumption on $|\theta|$ and
produces another non-positive power of $\cos\theta$.
On the other hand, it seems difficult to extend estimates for
$\|\me{-tH_\alpha}\|_{p\to p}$ (cf.~Corollary~\ref{lpboundedness}) to
complex times without any restrictions on $|\theta|$ or $V$.

\begin{corollary}
  \label{lpboundednesscomplex}
  Let $\alpha>0$, $\beta>0$, $1\leq p\leq2$, and $t>0$.
  
  (1) Suppose $\me{-tH_\alpha}$ satisfies
  the dyadic $(p,2,p')$ Davies--Gaffney estimate (Definition \ref{gge})
  or the dual dyadic $(p,p',p')$ Davies--Gaffney estimate
  (Definition \ref{ggerestricted})
  with $r\equiv r_t:=t^{1/\alpha}$ and $g(\lambda)\sim_\beta(1+\lambda)^{-\beta}$.
  Then, for $z=|z|\me{i\theta}\in\C_+$,
  $\zeta\geq0$, $r_z:=|z|^{1/\alpha}(\cos\theta)^{-\zeta}$,
  $\epsilon\in(0,1)$, and
  \begin{align*}
    \tilde\beta^{(2)}
    \equiv \tilde\beta_{d,\beta,p,\epsilon}^{(2)}(\theta)
    := \left(\beta-\frac{d}{p'}\right) \left(1-\frac{|\theta|}{\epsilon|\theta|+(1-\epsilon)\pi/2}\right) \geq0\,,
  \end{align*}
  one has
  \begin{align}
    \begin{split}
      \label{eq:lpboundednessggecomplex}
      \|\one_{B_x(r_z)}\me{-zH_\alpha}\one_{A_2(x,r_z,k)}\|_{p\to p}
      & \lesssim_{d,\alpha,p,\beta,\epsilon} C_\DG\, (\cos\theta)^{-d\left(\zeta+\frac{1}{\alpha}\right)\left(\frac12-\frac{1}{p'}\right)-\frac{d\zeta}{2}} 2^{-k\tilde\beta^{(2)}}
    \end{split}
  \end{align}
  for all $x\in\R^d$ and $k\in\N_0$.
  
  (2) If $\me{-tH_\alpha}$ satisfies the restricted dyadic $(p,2,p')$ or
  $(p,p',p')$ Davies--Gaffney estimate (Definition \ref{ggerestricted})
  and $|\theta|<\pi/2$ satisfies $\tilde\beta^{(2)}>d/p$, i.e.,
  \begin{align}
    \label{eq:condthetalpboundedness}
    \frac{|\theta|}{\epsilon|\theta|+(1-\epsilon)\pi/2} < 1 - \frac{d}{p}\left(\beta-\frac{d}{p'}\right)^{-1}\,,
  \end{align}
  then
  \begin{align}
    \label{eq:lpboundednessggecomplex2}
    \|\me{-zH_\alpha}\|_{p\to p} = \|\me{-\overline zH_\alpha}\|_{p'\to p'}
    \lesssim_{d,\alpha,p,\beta,\epsilon} C_\DG\, (\cos\theta)^{-d\left(\zeta+\frac{1}{\alpha}\right)\left(\frac12-\frac{1}{p'}\right)-\frac{d\zeta}{2}}\,.
  \end{align}
  
  (3) Suppose $V\in K_\alpha(\R^d)$ and $\mu_{\epsilon,V,d,\alpha}>0$ is
  the constant appearing in \eqref{eq:zhaozheng}.
  Then for all $z\in\C_+$ and $0<\epsilon\ll1$,
  \begin{align}
    \label{eq:lpboundednessggecomplex3}
    \|\me{-zH_\alpha}\|_{1\to 1}
    \lesssim_{d,\alpha,p} \me{\mu_{\epsilon,V,d,\alpha}|z|} (\cos\theta)^{-d(\frac1\alpha-\frac12)\one_{\{\alpha<1\}} -(\frac{d}{2}+\alpha-1)\one_{\{\alpha\geq1\}}}\,.
  \end{align}
\end{corollary}

\begin{remarks}
  (1) The power of $\cos\theta$ in \eqref{eq:lpboundednessggecomplex2}
  could not be correct if the estimate held for all $z\in\C_+$.
  This can be seen by considering $H_\alpha=-\Delta$, since
  $\|\me{z\Delta}\|_{1\to 1}\lesssim (\cos\theta)^{-\frac d2}$
  and $\|\me{z\Delta}\|_{2\to 2}\leq1$ imply
  $\|\me{z\Delta}\|_{p\to p}\lesssim(\cos\theta)^{-d(\frac12-\frac{1}{p'})}$
  for all $z\in\C_+$. Moreover, this upper bound is sharp as there is a
  matching lower bound, cf.~Arendt et al.~\cite[Lemma~2.2]{Arendtetal1997}.

  (2) Since $\beta>d$ in the case of restricted dyadic $(p,2,p')$ or
  $(p,p',p')$ Davies--Gaffney estimates, there exist $|\theta|\in[0,\pi/2)$
  for which $\tilde\beta^{(2)}>d/p$ is satisfied.
\end{remarks}

\begin{proof}
  To prove \eqref{eq:lpboundednessggecomplex} it suffices to assume that
  $\me{-tH_\alpha}$ satisfies the dyadic $(p,2,p')$ Davies--Gaffney estimate
  by Proposition~\ref{hypercontractivedual}.
  By H\"older's inequality and \eqref{eq:plgge} in Theorem~\ref{plgge}
  (which reduces to \eqref{eq:plggecor2} in this case) we have
  \begin{align*}
    & \|\one_{B_x(r_z)}\me{-zH_\alpha}\one_{A_2(x,r_z,k)}f\|_{p}
    \leq |B_x(r_z)|^{\frac{2-p}{2p}} \|\one_{B_x(r_z)}\me{-zH_\alpha}\one_{A_2(x,r_z,k)}f\|_2\\
    & \quad \lesssim_{d,\alpha,p,\beta,\epsilon} C_\DG\, r_z^{d\frac{2-p}{2p} - d\left(\frac1p-\frac12\right)} (\cos\theta)^{-d\left(\zeta+\frac{1}{\alpha}\right)\left(\frac12-\frac{1}{p'}\right)-\frac{d\zeta}{2}} 2^{-k\tilde\beta^{(2)}}\|f\|_p
  \end{align*}
  for any $f\in L^p$, $x\in\R^d$, and $k\in\N_0$. This proves
  \eqref{eq:lpboundednessggecomplex}.

  To prove \eqref{eq:lpboundednessggecomplex2} it suffices to assume
  that $\me{-tH_\alpha}$ satisfies the restricted dyadic $(p,2,p')$
  Davies--Gaffney estimate by Proposition~\ref{hypercontractivedual}.
  If $\tilde\beta^{(2)}>d/p$, then \eqref{eq:lpboundednessggecomplex}
  shows that $\me{-zH_\alpha}$ satisfies the (restricted) dyadic
  $(p,p,p'\cdot(1-|\theta|/\gamma_\epsilon)^{-1})$
  Davies--Gaffney estimate with 
  $c_{d,\alpha,p,\beta,\epsilon}C_\DG\,(\cos\theta)^{-d\left(\zeta+\frac{1}{\alpha}\right)\left(\frac12-\frac{1}{p'}\right)-\frac{d\zeta}{2}}$
  instead of $C_\DG$,
  and $\beta$ replaced by $\beta(1-|\theta|/\gamma_\epsilon)$, where
  $\gamma_\epsilon=\epsilon|\theta|+(1-\epsilon)\pi/2$.
  Thus, \eqref{eq:lpboundednessggecomplex2}
  follows from \eqref{eq:hypercontractive} in Proposition~\ref{hypercontractive}.
  Estimate \eqref{eq:lpboundednessggecomplex3} follows from \eqref{eq:zhaozheng}.
\end{proof}

\subsection{Applicability of the obtained heat kernel bounds}
\label{ss:applicability}

In this subsection we discuss the applicability of the
complex-time bounds obtained in the previous subsections.

\subsubsection{Regularization of Schr\"odinger groups}
\label{ss:schrodingerregularized}

A selection of applications of $L^p\to L^p$ bounds for complex-time heat kernels is
contained in \cite[Chapter~7]{Ouhabaz2005}. Here we focus on one of them, namely
$L^p\to L^p$-bounds of regularizations of Schr\"odinger groups.
The following abstract result is due to Boyadzhiev--de Laubenfels
\cite{BoyadzhievLaubenfels1993}. See also Elmennaoui \cite{Elmennaoui1992} for a
similar result concerning Riesz means.
In view of the ensuing sections it is stated in a slightly more general form
compared to the previous results.

\begin{theorem}[{\cite[Theorem~2.1]{BoyadzhievLaubenfels1993}}]
  \label{schrodingerregularized}
  Let $(\Omega,\mu)$ be a measure space, $p\in[1,\infty)$,
  $A\geq0$ a non-negative (and thereby self-adjoint) operator in $L^2(\Omega)$,
  and $\gamma>\delta\geq0$ such that
  $\|\me{-zA}\|_{p\to p}\lesssim_{\gamma,\delta,p}(\cos\theta)^{-\delta}$ for all
  $z=|z|\me{i\theta}\in\C_+$.
  If $\{\me{-zA}\}_{z\in\C_+}$ is bounded, strongly continuous, holomorphic
  of angle $\pi/2$ on $L^p(\Omega)$, then
  \begin{align}
    \label{eq:schrodingerregularized}
    \|(1+A)^{-\gamma}\me{itA}\|_{p\to p} \lesssim_{\gamma,\delta,p} (1+|t|)^\gamma\,, \quad t\in\R\,.  
  \end{align}
  Moreover, the map $\R\ni t\mapsto(1+A)^{-\gamma}\me{itA}$ is strongly
  continuous on $L^p(\Omega)$.
\end{theorem}

\begin{remark}
  As noted in \cite[p.~153]{Blunck2007}, the fact that $\{\me{-zA}\}_{z\in\C_+}$
  is strongly continuous on $L^p$ on all strict subsectors of $\C_+$ can be
  inferred from the assumptions in the theorem by arguing as in Ouhabaz
  \cite{Ouhabaz1995}, see also \cite[Corollary~7.5]{Ouhabaz2005}.
  The strong continuity of $\R\ni t\mapsto(1+A)^{-\gamma}\me{itA}$ follows from
  \eqref{eq:schrodingerregularized} and the strong continuity and holomorphy of
  $\{\me{-zA}\}_{z\in\C_+}$ as in \cite[p.~211]{Ouhabaz2005}.
\end{remark}

Theorem~\ref{schrodingerregularized} together with the pointwise bounds for
$\me{it(-\Delta)^{\alpha/2}}$ in \cite{ZhaoZheng2020} has the following
consequence.

\begin{corollary}
  Let $\alpha>0$, $p\in[1,\infty)$, and
  \begin{align*}
    \gamma >
    \begin{cases}
      2d(\frac1\alpha-\frac12)|1/p-1/2| & \quad \text{for}\ \alpha\in(0,1)\,,\\
      (d-1)|1/p-1/2| & \quad \text{for}\ \alpha=1\,,\\
      2(\frac d2+\alpha-1)|1/p-1/2| & \quad \text{for}\ \alpha>1\,.
    \end{cases}
  \end{align*}
  Then one has
  $\|(1+(-\Delta)^{\alpha/2})^{-\gamma}\me{it(-\Delta)^{\alpha/2}}\|_{p\to p}\lesssim_{\alpha,\gamma,d,p}(1+|t|)^\gamma$
  for all $t\in\R$.
\end{corollary}

\begin{proof}
  For $\alpha=1$, this is the content of \cite[Theorem~7.20]{Ouhabaz2005}.
  The claim for the other cases follows from Theorem~\ref{schrodingerregularized},
  interpolation with $\|\me{-z(-\Delta)^{\alpha/2}}\|_{2\to2}\leq1$, duality, and the
  bounds for $\me{-z(-\Delta)^{\alpha/2}}$ in \cite[Theorem~1.3]{ZhaoZheng2020},
  which yield
  \begin{align*}
    \|\exp(-z(-\Delta)^{\alpha/2})\|_{1\to 1}
    & \lesssim (\cos\theta)^{-d(\frac1\alpha-\frac12)\one_{\{\alpha<1\}} -(\frac{d}{2}+\alpha-1)\one_{\{\alpha\geq1\}}}\,.
  \end{align*}
  This concludes the proof.
\end{proof}

Unfortunately, the bounds in Corollary~\ref{lpboundednesscomplex}
are still too weak to apply Theorem~\ref{schrodingerregularized} to
$A=H_\alpha$ with $V\neq0$ due to the exponential growth in
\eqref{eq:lpboundednessggecomplex3} and the fact that the bound
\eqref{eq:lpboundednessggecomplex2} only holds for $z$ inside a sector strictly
contained in $\C_+$.

\textit{Remark:}
Chen et al.~\cite[Theorems~1.1, 1.3]{Chenetal2022} recently proved
\eqref{eq:schrodingerregularized}
for non-negative operators $A$ in $L^2(\R^d)$ under the assumption that
$\me{-tA}$ satisfies Poisson-type bounds like \eqref{eq:heatkernelv}
(with $\alpha>1+[d/2]$)
or a slight modification of the dyadic Davies--Gaffney estimate \eqref{eq:gge3}.
In particular, their results imply \eqref{eq:schrodingerregularized}
when $A=H_\alpha$ with $d=1$, $\alpha>1$, and $V\geq0$,
cf.~\cite[Section~1.3]{Chenetal2022}.

\subsubsection{Multiplier theorems}

The spectral theorem asserts that bounded and measurable functions of self-adjoint
operators in Hilbert spaces are bounded operators.
Proving a corresponding statement in Banach spaces is known to be much more delicate.
For instance, H\"ormander's classical multiplier theorem \cite{Hormander1960} asserts
the $L^p(\R^d)$ boundedness of Fourier multipliers $F(-\Delta)$ provided
the multiplier $F$ is sufficiently smooth.
(In fact, as Fefferman \cite{Fefferman1971} demonstrated, some smoothness of $F$ is
necessary for $L^p$-boundedness.)
If $F:[0,\infty)\to\C$ is a bounded, measurable function, then H\"ormander's theorem
asserts that the operator $F(-\Delta)$, which is initially defined via Plancherel's
theorem on $L^2(\R^d)$, extends to an $L^p(\R^d)$ bounded operator for all
$p\in(1,\infty)$ with 
\begin{align*}
  \|F(-\Delta)\|_{p\to p} \lesssim \sup_{t>0}\|\omega(\cdot) F(t\cdot)\|_{H^s(\R)}
\end{align*}
for any fixed non-trivial ``partition of unity'' function $\omega\in C_c^\infty(\R_+)$
which satisfies $\sum_{k\in\Z}\omega(2^kt)=1$ for all $t>0$, whenever $s>d/2$. 
For a selection of H\"ormander multiplier theorems for Schr\"odinger operators
$H_{\alpha=2}$ in $L^2(\R^d)$ we refer to \cite{Hebisch1990,Duongetal2002,Blunck2003}.

For Schr\"odinger operators $H_\alpha$ in $L^2(\R^d)$ with $\alpha\neq2$
there seem -- to the best of the author's knowledge -- only two results available.
Chen et al.~\cite[Section~5.3]{Chenetal2020} proved a multiplier theorem for
$H_\alpha$ in $L^2(\R^1)$ with $\alpha>1$ and $V\geq0$. On the other hand,
\cite[Theorem~2]{Merz2021} contains a H\"ormander multiplier theorem for $H_\alpha$
in $L^2(\R^d)$ with $\alpha<\min\{2,d\}$ and potentials $V(x)$ obeying
$\frac{a}{|x|^\alpha} \leq V(x) \leq \frac{\tilde a}{|x|^\alpha}$ for any $\tilde a\geq a>0$.
The latter result strongly relies on an abstract multiplier theorem by Hebisch
\cite{Hebisch1995} for operators whose heat kernels satisfy weighted
ultracontractive estimates and a certain H\"older condition which are tailored to
(smoothing) Poisson-type heat kernel bounds.

\medskip
Kriegler \cite[Corollary~3.6]{Kriegler2014} proved a H\"ormander multiplier theorem
for operators whose complex-time heat kernels satisfy Poisson-type bounds.
Translated to the language of the present work, Kriegler's theorem might apply to
non-negative operators $H_{\alpha=1}=\sqrt{-\Delta}+V$ in $L^2(\R^d)$ whose
complex-time heat kernel satisfies the following bound:  
there is $\beta\geq0$ such that for all $x,y\in\R^d$,
$z=|z|\me{i\theta}\in\C_+=\{z\in\C_+:\re(z)>0\}$, one has
\begin{align*}
  |\me{-zH}(x,y)| \lesssim (\cos\theta)^{-\beta}\frac{|z|}{|z^2+|x-y|^2|^{(d+1)/2}}\,.
\end{align*}
This assumption is satisfied when $V\equiv0$ and seems natural when one works with
$V\geq0$ or potentials $V$ whose negative part is not too singular, in particular
without Hardy singularities.
However, neither the bounds in Theorem~\ref{zhaozheng}, nor those in
Corollary~\ref{PLapplied2} suffice to conclude a multiplier theorem for $H_\alpha$
with $V\neq0$ using Kriegler's result because of the exponential growth
in \eqref{eq:zhaozheng} and the deteriorating decay in \eqref{eq:PLapplied2},
respectively.
%
%
We hope that this work stimulates further research in these directions.

\subsubsection{Operators on manifolds}

Li and Yau \cite{LiYau1986}, Davies \cite{Davies1988,Davies1990}, and
Sturm \cite{Sturm1992} proved pointwise sub-gaussian heat kernel bounds for the
negative of the Laplace--Beltrami operator $-\Delta_g\geq0$ on Riemannian manifolds
$(M,g)$ with lower bounded Ricci curvature.
Sturm \cite{Sturm1993} obtained sub-gaussian heat kernel bounds also for
$-\Delta_g+V$ when $V$ belongs to the Kato class.
Moreover, pointwise sub-gaussian heat kernel estimates are available for uniformly
elliptic second-order differential operators on domains in $\R^d$ with Dirichlet
boundary conditions, cf.~\cite[Theorem~6.10]{Ouhabaz2005}.
These estimates were extended to complex times, e.g., by Carron--Coulhon--Ouhabaz
\cite[Proposition~4.1]{Carronetal2002} (see also
\cite[Theorems~7.2--7.3]{Ouhabaz2005}) and used to prove
$L^p\to L^p$-estimates for complex-time heat kernels, cf.~\cite[Theorem~4.3]{Carronetal2002}
or \cite[Theorem~7.4]{Ouhabaz2005}.
As discussed in Subsection~\ref{ss:schrodingerregularized}, these estimates lead,
among others, to $L^p\to L^p$-bounds for regularizations of the corresponding
Schr\"odinger group, cf.~\cite[Theorem~5.2]{Carronetal2002} or
\cite[Theorem~7.12]{Ouhabaz2005}.

\medskip
It is natural to consider analogous questions for $(-\Delta_g)^{\alpha/2}+V$, where
$(-\Delta_g)^{\alpha/2}$ is defined by the spectral theorem,
or for fractional Schr\"odinger operators on domains $\Omega\subseteq\R^d$.
We first discuss the latter.
For sufficiently regular $V\in L_{\rm loc}^1(\Omega:\R)$ (e.g.,
$V\in L_{\rm loc}^{d/\alpha}$ suffices), and $\psi$ belonging to the Sobolev space
$H^\alpha(\R^d)$ and vanishing almost everywhere in $\R^d\setminus\Omega$, the
quadratic form $\langle\psi,((-\Delta)^{\alpha/2}+V)\psi\rangle_{L^2(\R^d)}$ is
bounded from below and closed in the Hilbert space $L^2(\Omega)$.
Thus, this form generates a self-adjoint operator $H_\alpha^{(\Omega)}$ in $L^2(\Omega)$,
which is also bounded from below.
For $0\leq V\in L_{\rm loc}^1(\Omega)$ the Friedrichs extension automatically
provides us a self-adjoint operator, whose heat kernel can be bounded using the
maximum principle by
\begin{align*}
  \exp(-tH_\alpha^{(\Omega)})(x,y) \leq \exp(-t(-\Delta)^{\alpha/2})(x,y)\,, \quad x,y\in\Omega\,,
\end{align*}
whenever $\alpha\in(0,2)$.
Since such bounds are the only input in the proof of Corollary~\ref{PLapplied2},
one obtains analogous complex-times heat kernel estimates.
\begin{corollary}
  \label{PLappliedDomains}
  Let $\alpha\in(0,2)$, $\Omega\subseteq\R^d$ be an open subset and
  $0\leq V\in L_{\rm loc}^1(\Omega)$.
  Let further $z=|z|\me{i\theta}$ with $|\theta|\in[0,\pi/2)$,
  $x,y\in\Omega$, $\epsilon\in(0,1)$, and
  $\beta_{d,\alpha,\epsilon}(\theta)$ be defined as in Corollary~\ref{PLapplied2}.
  Then
  \begin{align}
    \label{eq:PLappliedDomains}
    \begin{split}
      |\exp(-zH_\alpha^{(\Omega)})(x,y)|
      \lesssim (|z|\cos\theta)^{-\frac{d}{\alpha}}\left(1+\frac{|x-y|}{|z|^{1/\alpha}}\right)^{-\beta_{d,\alpha,\epsilon}(\theta)}\,.
    \end{split}
  \end{align}
\end{corollary}

\medskip
Regarding fractional powers of $-\Delta_g$ on compact $d$-dimensional Riemannian
manifolds $(M,g)$, Gimperlein and Grubb \cite[Theorem~4.2]{GimperleinGrubb2014}
exploited the subordination principle
(cf.~\cite[Chapter~5]{Schillingetal2012}, see also
\cite{Pollard1946,Hawkes1971,IbragimovLinnik1971} for sharp bounds and
\cite{Pollard1946,Bergstrom1952}
for explicit expressions of the subordinator)
to prove estimates for $\exp(-t(-\Delta_g)^{\alpha/2})$ with $\alpha\in(0,2)$
using those for $\me{t\Delta_g}$.
For $x,y\in M$ and $t>0$ they obtained
\begin{align}
  \label{eq:gimperlein1}
  \me{-t(-\Delta_g)^{\alpha/2}}(x,y)
  \sim_{d,\alpha} \frac{t}{(\rho(x,y)+t^{1/\alpha})^\alpha}\cdot\left(1+(\rho(x,y)+t^{1/\alpha})^{-d}\right)\,,
\end{align}
where $\rho(x,y)$ denotes the geodesic distance between $x$ and $y$.
Moreover, they extended these estimates to complex times
\cite[Theorem~1]{GimperleinGrubb2014} and showed for $z\in\C_+$,
\begin{align}
  \label{eq:gimperleincomplex1}
  |\me{-z(-\Delta_g)^{\alpha/2}}(x,y)|
  \lesssim_{d,\alpha} (\cos\theta)^{-N}\frac{|z|}{(\rho(x,y)+|z|^{\frac1\alpha})^\alpha}\cdot\left(1+(\rho(x,y)+|z|^{\frac1\alpha})^{-d}\right)
\end{align}
with $N=\max\{d/\alpha,7d/2+4\alpha+7\}$.
In \cite[Theorem~4.3]{GimperleinGrubb2014} they also obtained real-time
heat kernel estimates for $(-\Delta_g)^{\alpha/2}+V$ with $V$ being any,
not necessarily self-adjoint, classical pseudodifferential operator of
order $\alpha-1$.
Assuming additionally from now on that $V$ is such that
$(-\Delta_g)^{\alpha/2}+V\geq0$, then their result reads
\begin{align}
  \label{eq:gimperlein2}
  \begin{split}
    |\me{-t((-\Delta_g)^{\alpha/2}+V)}(x,y)|
    & \lesssim_{d,\alpha} \frac{t}{(\rho(x,y)+t^{1/\alpha})^\alpha}\cdot\left(1+(\rho(x,y)+t^{1/\alpha})^{-d}\right)\\
    & \quad + \frac{t}{(\rho(x,y)+t^{1/\alpha})^{d+\alpha-1}}\,, \quad t>0\,.
  \end{split}
\end{align}
Since $(-\Delta_g)^{\alpha/2}+V$ is self-adjoint in this situation,
we can estimate
\begin{align}
  \label{eq:gimperlein3}
  \begin{split}
    |\me{-2z((-\Delta_g)^{\alpha/2}+V)}(x,y)|
    & \leq \|\me{-\re(z)((-\Delta_g)^{\alpha/2}+V)}\|_{1\to2}^2\\
    & \lesssim_{d,\alpha} 1+(|z|\cos\theta)^{-\frac d\alpha} + (|z|\cos\theta)^{-\frac{d-1}{\alpha}}\\
    & \lesssim 1+(|z|\cos\theta)^{-\frac d\alpha}
  \end{split}
\end{align}
for $z=|z|\me{i\theta}\in\C_+$.
On the other hand, since $M$ is compact, \eqref{eq:gimperlein2} implies
\begin{align}
  \label{eq:gimperlein4}
  |\me{-|z|((-\Delta_g)^{\alpha/2}+V)}(x,y)|
  \leq c_{d,\alpha,M} |z|^{-\frac{d}{\alpha}} \left(\frac{\rho(x,y)^{\alpha}}{|z|}\right)^{-\frac{d+\alpha}{\alpha}},
\end{align}
where $c_{d,\alpha,M}>0$ only depends on $d,\alpha$, and $\sup_{x,y\in M}\rho(x,y)<\infty$.
We now prove a variant of Theorem~\ref{PL} to obtain further estimates for
$|\me{-z((-\Delta_g)^{\alpha/2}+V)}(x,y)|$.

\begin{lemma}
  \label{PLGimperlein}
  Let $X$ be a Banach space equipped with a norm $\|\cdot\|$ and
  $F:\C_+=\{z\in\C:\,\re(z)>0\}\to X$ be a holomorphic function satisfying
  \begin{align}
    \label{eq:assum1gimperlein}
    \|F(|z|\me{i\theta})\|
    & \leq 1 + a_1(|z|\cos\theta)^{-\beta_1}
      \quad \text{and}\\
    \label{eq:assum2gimperlein}
    \|F(|z|)\|
    & \leq a_1|z|^{-\beta_1}\left(\frac{a_2}{|z|}\right)^{-\beta_2}\cdot \left(\frac{a_3}{|z|}\right)^{\beta_3}
  \end{align}
  for some $a_1,a_2,a_3>0$,
  $\beta_1,\beta_2,\beta_3\geq0$,
  all $|z|>0$, and all $|\theta|<\pi/2$.
  Then for all $\epsilon\in(0,1)$ one has
  \begin{align}
    \label{eq:pl1gimperlein}
    \begin{split}
      \|F(|z|\me{i\theta})\|
      & \leq 2\left|1+\frac{a_1^{1/\beta_1}}{z}\right|^{\beta_1} (\epsilon \cos\theta)^{-2\beta_1}\cdot \left( \left(\frac{a_2}{|z|}\right)^{-\beta_2} \cdot \left(\frac{a_3}{|z|}\right)^{\beta_3}\right)^{1-|\theta|/\gamma(\epsilon,\theta)}
    \end{split}
  \end{align}
  for all $|\theta|<\pi/2$ and $|z|>0$, where
  $\gamma(\epsilon,\theta):=\epsilon|\theta|+(1-\epsilon)\pi/2$.
\end{lemma}

\begin{proof}
  The proof is similar to that of Theorem~\ref{PL}, so we merely indicate
  the needed modifications.
  We define the functions $H_2(z)$ and $H_3(z)$ as in
  \eqref{eq:choicePL}--\eqref{eq:choicePL2},
  while we choose the function $G(z)$ as
  \begin{align*}
    G(z) := (1+a_1^{1/\beta_1}z)^{-\beta_1} \cdot F(z^{-1}) \cdot H_{2}(z) \cdot H_{3}(z)\,.
  \end{align*}
  Then we have $\|G(|z|)\|\leq 1$ similarly as before,
  but for $\gamma\in(0,\pi/2)$, we obtain
  \begin{align*}
    \|G(|z|\me{i\gamma})\|
    & \leq \frac{1+a_1(|z|^{-1}\cos\gamma)^{-\beta_1}}{|1+a_1^{1/\beta_1}|z|\me{i\gamma}|^{\beta_1}}
      \leq \frac{(\cos\gamma)^{-\beta_1}(1+a_1|z|^{\beta_1})}{(1+a_1^{1/\beta_1}|z|\cos\gamma)^{\beta_1}}
      \leq 2(\cos\gamma)^{-2\beta_1}\,.
  \end{align*}
  Proceeding as in the proof of Theorem~\ref{PL}, we obtain
  for $-\gamma\leq\theta<0$,
  \begin{align*}
    \|F(|z|\me{i\theta})\|
    \leq 2\left|1+\frac{a_1^{1/\beta_1}}{z}\right|^{\beta_1} \cdot \epsilon^{-2\beta_1}(\cos\theta)^{-2\beta_1} \cdot \left(\left(\frac{a_2}{|z|}\right)^{-\beta_2} \cdot \left(\frac{a_3}{|z|}\right)^{\beta_3}\right)^{1+\theta/\gamma(\epsilon,\theta)}\,.
  \end{align*}
  Reflecting the estimate along the real axis yields \eqref{eq:pl1gimperlein}.
\end{proof}

This lemma and the bounds \eqref{eq:gimperlein3}--\eqref{eq:gimperlein4} yield,
as in the proof of Corollary~\ref{PLapplied2}, the following estimates.

\begin{corollary}
  \label{PLappliedgimperlein}
  Let $\alpha\in(0,2)$, $\epsilon\in(0,1)$, and
  $(M,g)$ be a $d$-dimensional compact Riemannian manifold
  with geodesic distance $\rho:M\times M\to\R_+$.
  Let $V$ be a classical pseudodifferential operator of order $\alpha-1$
  (in the sense of \cite{GimperleinGrubb2014})
  such that $(-\Delta_g)^{\alpha/2}+V \geq0$ in $L^2(M)$.
  Then for all $x,y\in M$ estimate \eqref{eq:gimperlein2} holds and,
  for all $z=|z|\me{i\theta}\in\C_+$, and $\beta_{d,\alpha,\epsilon}(\theta)$
  as in\eqref{eq:PLapplied2defbeta} in Corollary~\ref{PLapplied2}, we have
  \begin{align}
    \label{eq:PLappliedgimperlein}
    \begin{split}
      & |\me{-z((-\Delta_g)^{\alpha/2}+V)}(x,y)|\\
      & \quad \leq \! c_{d,\alpha,M}\min\left\{1\!+\!(|z|\cos\theta)^{-\frac{d}{\alpha}}, (\epsilon \cos\theta)^{-\frac{2d}{\alpha}}\left(1\!+\!|z|^{-1}\right)^{\frac d\alpha} \cdot \left(\frac{\rho(x,y)}{|z|^{1/\alpha}}\right)^{-\beta_{d,\alpha,\epsilon}(\theta)}\right\},
    \end{split}
  \end{align}
  where the constant $c_{d,\alpha,M}>0$ only depends on
  $d,\alpha$, and $\sup_{x,y\in M}\rho(x,y)<\infty$.
\end{corollary}

When $\rho(x,y)<|z|^{1/\alpha}$, then the bound \eqref{eq:gimperlein3} for
$\me{-z((-\Delta_g)^{\alpha/2}+V)}(x,y)$ is already suitable.
On the other hand, the estimate
\begin{align*}
  \left(1+|z|^{-1}\right)^{\frac d\alpha} \cdot \left(\frac{\rho(x,y)}{|z|^{1/\alpha}}\right)^{-(d+\alpha)}
  \lesssim \frac{|z|}{\rho(x,y)^{d+\alpha}} + \frac{|z|^{1+\frac d\alpha}}{\rho(x,y)^{d+\alpha}}
  \leq \frac{|z|}{\rho(x,y)^{d+\alpha}} + \frac{|z|}{\rho(x,y)^{\alpha}}
\end{align*}
shows that -- disregarding different negative powers of the prefactor $\cos(\theta)$
and pretending that $\beta_{d,\alpha,\epsilon}(\theta)$ could be replaced by $d+\alpha$ in
\eqref{eq:PLappliedgimperlein} -- the right sides of \eqref{eq:PLappliedgimperlein}
and \eqref{eq:gimperleincomplex1} ``qualitatively agree'' when $\rho(x,y)>|z|^{1/\alpha}$.


\section{Consequences of dyadic Davies--Gaffney estimates}
\label{ss:consequencesgge}

We collect some consequences of the dyadic Davies--Gaffney estimate
\eqref{eq:gge3} (Definitions \ref{gge} and \ref{ggerestricted}).
The dyadic partition $1=\sum_{k\geq0}\one_{A_2(x,r,k)}$ for any
$x\in\R^d$ and $r>0$, and the triangle inequality yield the following
preliminary estimate.

\begin{corollary}
  \label{eqdyadic}
  Let $r>0$, $(T_r)_{r>0}$ be a family of linear operators that satisfy
  the dyadic $(p,q,\sigma)$ Davies--Gaffney estimate \eqref{eq:gge3}
  (Definition \ref{gge}). Then
  \begin{align}
    \label{eq:eqdyadic}
    \|\one_{B_x(r)}T_r\|_{p\to q} \lesssim_{d,\beta,\sigma} C_{\DG}\, r^{-d\left(\frac1p-\frac1q\right)}\,,
    \quad x\in\R^d\,.
  \end{align}
\end{corollary}

In fact, for \eqref{eq:eqdyadic} to hold, one only needs
$\sum_{k\geq0}g(k)2^{kd/\sigma}<\infty$.
We will upgrade this estimate in Proposition~\ref{hypercontractive}
with the help of the ball averages
\begin{align}
  \label{eq:defmaximalp}
  (N_{p,r}f)(x) := r^{-\frac dp}\|\one_{B_x(r)} f\|_{L^p(\R^d)}
  \quad \text{and} \quad
  (N_{p,q,r}f)(x) := r^{-\frac dq}\|\one_{B_x(r)} f\|_{L^p(\R^d)}
\end{align}
for $p,q\in[1,\infty)$.
The following lemma by Blunck and Kunstmann
summarizes useful estimates involving these averaging operators.

\begin{lemma}[{\cite[Lemma~3.3]{BlunckKunstmann2005}}]
  \label{collectionmaximal}
  Let $1\leq p\leq q\leq\infty$ and $r>0$.
  Then 
  \begin{enumerate}
  \item $(N_{p,r}f)(x) \lesssim_{p,q} (N_{q,r}f)(x)$ for all $x\in\R^d$,
  \item $\|f\|_{L^p(\R^d)} \sim_p \|N_{p,r}f\|_{L^p(\R^d)}$, and
  \item $\|N_{p,q,r}f\|_{L^q(\R^d)} \lesssim_{p,q} \|f\|_{L^p(\R^d)}$.
  \end{enumerate}  
\end{lemma}


\begin{proposition}
  \label{hypercontractive}
  Let $r>0$ and $(T_r)_{r>0}$ be a family of linear operators that satisfy
  the dyadic $(p,q,\sigma)$ Davies--Gaffney estimate \eqref{eq:gge3}
  (Definition \ref{gge}). Then
  \begin{align}
    \label{eq:hypercontractive}
    \|T_r\|_{p\to q} \lesssim_{d,\beta,\sigma,p,q} C_{\DG}\, r^{-d\left(\frac1p-\frac1q\right)}\,.
  \end{align}
  If $q=p'$ (so in particular $1\leq p\leq2$)
  and $T_r$ is in addition self-adjoint in $L^2(\R^d)$ and
  satisfies $T_{r+s}=T_rT_s$ for all $r,s>0$, then
  \begin{align}
    \label{eq:dualityheat}
    \|T_r\|_{p\to 2} = \|T_r\|_{2\to p'}
    \lesssim_{d,\beta,\sigma,p} C_{\DG}^{1/2} r^{-d\left(\frac12-\frac{1}{p'}\right)}\,.
  \end{align}
\end{proposition}

\begin{proof}
  By $\|f\|_q\sim\|N_{q,r}f\|_q$, a dyadic partition of unity,
  assumption \eqref{eq:gge3}, and $\|N_{p,q,r}f\|_q\lesssim \|f\|_p$
  for $1\leq p\leq q\leq\infty$, we have
  \begin{align*}
    \|T_rf\|_q
    & \lesssim_q \|N_{q,r}T_rf\|_q
      = \left(\int_{\R^d}\left(r^{-d/q}\|\one_{B_x(r)}T_r\sum_{k=0}^\infty \one_{A_2(x,r,k)}f\|_{q} \right)^q\,dx\right)^{1/q}\\
    & \leq C_{\DG}\, r^{-d\left(\frac1p-\frac1q\right)}\left( \int_{\R^d}\left( \sum_{k=0}^\infty g(2^k)2^{kd(\frac1\sigma+\frac1q)} (2^kr)^{-\frac{d}{q}}\|\one_{B_x(2^kr)}f\|_p \right)^q\,dx \right)^{1/q}\\
    & \leq C_{\DG}\, r^{-d\left(\frac1p-\frac1q\right)}\sum_{k=0}^\infty g(2^k)2^{kd(\frac1\sigma+\frac1q)} \|N_{p,q,2^kr}f\|_q\\
    & \lesssim_{p,q} C_{\DG}\, r^{-d\left(\frac1p-\frac1q\right)}\sum_{k=0}^\infty g(2^k)2^{kd(\frac1\sigma+\frac1q)}\|f\|_p
      \lesssim_{d,\sigma,\beta,q} C_{\DG}\, r^{-d\left(\frac1p-\frac1q\right)}\|f\|_p\,.
  \end{align*}
  This shows \eqref{eq:hypercontractive}.
  Formula \eqref{eq:dualityheat} follows from
  $\|T_{2r}\|_{p\to p'}=\|T_r\|_{p\to2}^2=\|T_r\|_{2\to p'}^2$.
\end{proof}

We need one other consequence of \eqref{eq:gge3} in the proof of
Theorem~\ref{plgge}.
To that end we record the following geometric observations.

\begin{lemma}
  \label{geomannuli}
  Let $x,z\in\R^d$, $r,r_0>0$, and $|x-z|\leq r+r_0$.
  \begin{enumerate}
  \item If $j,k\in\N\setminus\{1\}$ and
    $A_2(z,r,j)\cap A_2(x,r_0,k)\neq\emptyset$,
    then $2^{k-3}r_0\leq 2^jr\leq 2^{k+3}r_0$.
    
  \item If $j\in\N$ and $A_2(z,r,j)\cap B_x(r_0)\neq\emptyset$,
    then $(2^{j-1}-1)r \leq 2r_0$, i.e., $j\leq 1+\log_2(1+2r_0/r)$.
  %
  %
  %
  %
  \end{enumerate}
\end{lemma}

\begin{proof}
  By a translation we can assume $z=0$ without loss of generality.
  In that case $|x|\leq r+r_0$.
  Let $y\in\R^d$ denote those points belonging to the intersection of the
  geometric objects in claims (1) and (2). We use that, if
  $a\leq b$ and $c\leq d$ and, if $[a,b]\cap[c,d]\neq\emptyset$, then either
  $a\leq c\leq b \leq d$, or
  $c\leq a\leq d\leq b$, or
  $a\leq c \leq d\leq b$, or
  $c\leq a\leq b\leq d$.
  If in addition $a=0$ and $b,c,d\geq0$, then $c\leq b$.
  \begin{enumerate}
  \item If the annuli intersect, then
    $|y|\in[2^{j-1}r,2^j r]$ and $|x-y|\in[2^{k-1}r_0,2^k r_0]$.
    By the triangle inequality and the bounds on $|x|$ and $|y|$, we also have
    \begin{align*}
      |x-y| & \leq |x|+|y| \leq r_0 + (2^j+1) r \quad \text{and}\\
      |x-y| & \geq \max\{|y|-|x|,0\} \geq \max\{(2^{j-1}-1)r - r_0 , 0 \}\,.
    \end{align*}
    Thus,
    $|x-y|\in[2^{k-1}r_0,2^k r_0]\cap\left[\max\{(2^{j-1}-1)r - r_0 , 0 \},(2^j+1)r + r_0\right]$.
    This shows that, if $(2^{j-1}-1)r > r_0$, then either
    \begin{enumerate}
    \item $2^{k-1}r_0 \leq (2^{j-1}-1)r-r_0 \leq 2^k r_0\leq (2^{j}+1)r+r_0$, or
    \item $(2^{j-1}-1)r-r_0 \leq 2^{k-1}r_0 \leq (2^j +1)r+r_0\leq 2^k r_0$, or
    \item $2^{k-1}r_0\leq (2^{j-1}-1)r - r_0 \leq (2^j + 1)r + r_0 \leq 2^k r_0$, or
    \item $(2^{j-1}-1)r - r_0\leq 2^{k-1}r_0 \leq 2^k r_0 \leq (2^j+1)r + r_0$.
    \end{enumerate}
    In particular, $2^{k-3}r_0\leq 2^j r\leq 2^{k+3}r_0$.
    If instead $(2^{j-1}-1)r \leq r_0$, then
    $(2^{k-1}-1)r_0\leq (2^j+1)r$, so in particular
    $2^{k-3}r_0\leq 2^{j}r \leq 2^{k+3}r_0$.

  \item If $y\in A_2(0,r,j) \cap B_x(r_0)\neq\emptyset$, then
    $|y|\in[2^{j-1}r,2^j r]$
    and $|x-y|\leq r_0$. The triangle inequality and the bounds on
    $|x-y|$ and $|x|$ imply $|y| = |y-x+x| \leq 2r_0 + r$ and therefore
    $2^{j-1}r\leq 2r_0+r$. 
    %
    %
  \end{enumerate}
\end{proof}

\begin{proposition}
  \label{eqgge}
  Let $r>0$ and $(T_r)_{r>0}$ be a family of linear operators that satisfy
  the dyadic $(p,q,\sigma)$ Davies--Gaffney estimate \eqref{eq:gge3}
  (Definition \ref{gge}). 
  Then
  \begin{align}
    \label{eq:gge3eq}
    \begin{split}
      & \|\one_{B_x(r_0)} T_r\one_{A_2(x,r_0,k)}\|_{p\to q}\\
      & \quad \lesssim_{d,\sigma,\beta,p,q} C_{\DG}\, r^{-d\left(\frac1p-\frac1q\right)}\,\left(1 + \frac{|B_x(r_0)|}{r^d}\right)^{1/q}\\
      & \qquad \times \left[\left(\frac{d(B_x(r_0),A_2(x,r_0,k))}{r}\right)^{-\beta} \! \cdot \! \left(\frac{|A_2(x,r_0,k)|}{r^d}\right)^{\! \frac1\sigma}\theta(k\!-\!2) \!+\! \theta(1\!-\!k)\right]
    \end{split}
  \end{align}
  holds for all $x\in\R^d$, $r_0>0$, and $k\in\N_0$.
  In fact, the factor $(1+|B_x(r_0)|/r^d)^{1/q}$ can be
  removed for $k\in\{0,1\}$, i.e.,
  \begin{align}
    \label{eq:gge3eqimproved}
    \begin{split}
      & \|\one_{B_x(r_0)} T_r\one_{A_2(x,r_0,k)}\|_{p\to q}\\
      & \quad \lesssim_{d,\sigma,\beta,p,q} C_{\DG}\, r^{-d\left(\frac1p-\frac1q\right)}\!\left[\! \left(\frac{d(B_x(r_0),A_2(x,r_0,k))}{r}\right)^{-\beta} \! \cdot \! \left(\frac{|A_2(x,r_0,k)|}{r^d}\right)^{\! \frac1\sigma} \right.\\
      & \qquad\qquad\qquad\qquad\qquad\qquad  \left. \times \left(1 + \frac{|B_x(r_0)|}{r^d}\right)^{\frac1q}\theta(k-2) + \theta(1-k) \right]
    \end{split}
  \end{align}
  holds for all $x\in\R^d$, $r_0>0$, and $k\in\N_0$.
\end{proposition}

Recall that \eqref{eq:gge3} only involved one radius.
Here the radii $r$ appearing in $T_r$,
and $r_0$ appearing in $B_x(r_0)$ and $A_2(x,r_0,k)$
in formula \eqref{eq:gge3eq} are independent of each other.
Note also that the proof of \eqref{eq:gge3eq} only needs $\beta>d/\sigma$.

\begin{proof}
  To prove \eqref{eq:gge3eq} we use $\|f\|_q\sim \|N_{q,r}f\|_q$,
  decompose dyadically, apply \eqref{eq:gge3}, and obtain
  \begin{align}
    \label{eq:anyr0}
    \begin{split}
      & \|\one_{B_x(r_0)}T_r\one_{A_2(x,r_0,k)}f\|_q
      \lesssim_q \|N_{q,r}(\one_{B_x(r_0)}T_r\one_{A_2(x,r_0,k)}f)\|_q\\
      & \quad = \left(\! \int_{\R^d}\left(\! r^{-\frac dq}\|\one_{B_z(r)}\one_{B_x(r_0)}T_r \sum_{j=0}^\infty \one_{A_2(z,r,j)}\one_{A_2(x,r_0,k)}f\|_q \!\right)^q\,dz \!\right)^{1/q}\\
      & \quad \leq C_{\DG}\, r^{-\frac dp}\sum_{j=0}^\infty g(2^j)2^{\frac{jd}{\sigma}} \left(\int\limits_{|z-x|\leq r+r_0} \|\one_{A_2(z,r,j)}\one_{A_2(x,r_0,k)}f\|_p^q\,dz\right)^{1/q}\,.
    \end{split}
  \end{align}
  To estimate $(...)^{1/q}$ on the right side we use
  Lemma~\ref{geomannuli} (with $|x-z|\leq r+r_0$) and
  distinguish between the following cases.
  \begin{enumerate}
  \item If $j,k\in\N\setminus\{1\}$, and
    $A_2(z,r,j)\cap A_2(x,r_0,k)\neq\emptyset$,
    then $2^{k-3}r_0 \leq 2^{j}r \leq 2^{k+3}r_0$
    by (1) in Lemma~\ref{geomannuli}.
  \item If $k\in\N\setminus\{1\}$, $j\in\{0,1\}$, and
    $A_2(z,r,j)\cap A_2(x,r_0,k)\neq\emptyset$, then
    $B_z(2^jr)\cap A_2(x,r_0,k)\neq\emptyset$
    and so $2^jr\geq 2^{k-3}r_0$
    by (2) in Lemma~\ref{geomannuli}.
  \item If $k\in\{0,1\}$, we put no restriction on $j\in\N_0$.
  \end{enumerate}
  Thus,
  \begin{align*}
    & \|\one_{A_2(z,r,j)}\one_{A_2(x,r_0,k)}f\|_p\\
    & \quad \lesssim \|f\|_p\left[\one_{\{2^j r \geq 2^{k-3}r_0\}}\theta(k-2)\left(\one_{\{2^j r\leq2^{k+3}r_0\}}\theta(j-2)+\theta(1-j)\right) + \theta(1-k)\right]\,.
  \end{align*}
  If $k\in\{0,1\}$, we simply sum over all $j\in\N_0$.
  If $k\in\N\setminus\{1\}$ and $j\in\N\setminus\{1\}$, we use
  $g(2^j)2^{jd/\sigma}\lesssim_\beta 2^{-j(\beta-d/\sigma)}$
  and that we are summing over dyadic numbers.
  If $k\in\N\setminus\{1\}$ and $j\in\{0,1\}$, we use
  $\beta\geq d/\sigma$ and $2^jr\geq 2^{k-3}r_0$ to estimate
  $g(2^j)\cdot 2^{\frac{jd}{\sigma}} \lesssim_{d,\beta,\sigma} (2^kr_0/r)^{-\beta+d/\sigma}$.
  Thus, \eqref{eq:anyr0} can be estimated by  
  \begin{align*}
    & \|\one_{B_x(r_0)}T_r \one_{A_2(x,r_0,k)}f\|_q\\
    & \quad \lesssim_q C_{\DG}\, r^{-\frac dp} \Big(\int\limits_{|z-x|\leq r+r_0} dz\Big)^{1/q}\|f\|_p\\
    & \times \sum_{j=0}^\infty \!\left[ 2^{j(\frac{d}{\sigma}-\beta)} \! \left(\one_{\{2^j r \geq 2^{k-3}r_0\}}\theta(k-2)(\one_{\{2^j r\leq 2^{k+3}r_0\}}\theta(j-2)+\theta(1-j)) + \theta(1-k)\right)\right]\\
    & \quad \lesssim_{d,\sigma,\beta,q} C_{\DG}\, r^{-d(\frac1p-\frac1q)} \cdot \left(1+\frac{r_0}{r}\right)^{\frac dq} \left[\left(\frac{2^kr_0}{r}\right)^{-\beta+d/\sigma}\theta(k-2)+\theta(1-k)\right]\|f\|_p\,.
  \end{align*}
  Since $2^kr_0\sim d(B_x(r_0),A_2(x,r_0,k))\sim |A_2(x,r_0,k)|^{\frac1d}$
  for $k\geq2$, this proves \eqref{eq:gge3eq}.
  The improved estimate \eqref{eq:gge3eqimproved} for $k\in\{0,1\}$
  follows from 
  $\|\one_{B_x(r_0)} T_r\one_{A_2(x,r_0,k)}\|_{p\to q}\leq\|T_r\|_{p\to q}$
  and Proposition~\ref{hypercontractive}.
  This concludes the proof of Proposition~\ref{eqgge}.
\end{proof}

The proof of Corollary~\ref{plggecor} relies on the following
proposition which says that $(p,p',p')$ estimates imply $(2,p',2)$
and $(p,2,p')$ estimates, if one assumes additionally
$\sum_{k\geq0}g(2^k)2^{kd(\frac{1}{p'}+\frac12)}<\infty$.
This assumption is contained in the notion of dual dyadic
Davies--Gaffney estimates.

\begin{proposition}
  \label{hypercontractivedual}
  Let $p\in[1,2]$ and $r>0$.

  (1) If $(T_r)_{r>0}$ is a family of linear operators
  that satisfy the dual dyadic $(p,p',p')$ Davies--Gaffney
  estimate \eqref{eq:gge3} (Definition \ref{ggerestricted}),
  then $T_r$ satisfies the dyadic $(2,p',2)$ Davies--Gaffney
  estimate (Definition \ref{gge}), i.e.,
  \begin{subequations}
    \begin{align}
      \label{eq:hypercontractivedual}
      \|\one_{B_x(r)}T_r\one_{A_2(x,r,k)}\|_{2\to p'}
      \lesssim_{d,\beta,p} C_\DG\, r^{-d\left(\frac12-\frac{1}{p'}\right)}g(2^k)2^{\frac{kd}{2}}\,,
      \quad x\in\R^d,k\in\N_0
    \end{align}
    and the dyadic $(p,2,p')$ Davies--Gaffney estimate, i.e.,
    \begin{align}
      \label{eq:hypercontractivedual2}
      \|\one_{B_x(r)}T_r\one_{A_2(x,r,k)}\|_{p\to 2}
      \lesssim_{d,\beta,p} C_\DG\, r^{-d\left(\frac1p-\frac12\right)}g(2^k)2^{\frac{kd}{p'}}\,,
      \quad x\in\R^d,k\in\N_0\,.
    \end{align}
  \end{subequations}

  (2) If $(T_r)_{r>0}$ satisfies the restricted dyadic $(p,p',p')$
  Davies--Gaffney estimate \eqref{eq:gge3} (Definition \ref{ggerestricted}),
  then it satisfies the restricted dyadic $(2,p',2)$ and $(p,2,p')$ Davies--Gaffney
  estimates \eqref{eq:hypercontractivedual} and \eqref{eq:hypercontractivedual2}.
\end{proposition}

\begin{proof}
  Since $\beta>d(1/2+1/p')$ for the dual dyadic $(p,p',p')$ estimates
  and $\beta>d$ for the restricted dyadic $(p,p',p')$ estimates, it suffices
  to show \eqref{eq:hypercontractivedual} and \eqref{eq:hypercontractivedual2}.
  
  To prove \eqref{eq:hypercontractivedual} we use
  $\|f\|_{p'} \sim_p \|N_{p',r}f\|_{p'}$
  and obtain for $f\in L^2$,
  \begin{align*}
    & \|\one_{B_x(r)}T_r\one_{A_2(x,r,k)}f\|_{p'}
    \lesssim_p \|N_{p',r}\one_{B_x(r)}T_r\one_{A_2(x,r,k)}f\|_{p'}\\
    & \quad = \left( \int_{\R^d}dz \left( r^{-\frac{d}{p'}}\|\one_{B_z(r)}\one_{B_x(r)}T_r \one_{A_2(x,r,k)}\sum_{j\geq0}\one_{A_2(z,r,j)}f\|_{p'} \right)^{p'} \right)^{1/p'}\\
    & \quad \leq C_\DG\, r^{-\frac{d}{p}}\sum_{j\geq0}g(2^j)2^{\frac{jd}{p'}} \left(\int\limits_{|x-z|\leq2r}dz\ \|\one_{A_2(x,r,k)}\one_{A_2(z,r,j)}f\|_{p}^{p'}\right)^{1/p'}\,.
  \end{align*}
  We use Lemma~\ref{geomannuli} (with $|x-z|\leq 2r$) and
  distinguish between the following cases.
  \begin{enumerate}
  \item If $k,j\in\N\setminus\{1\}$ and
    $A_2(x,r,k)\cap A_2(z,r,j)\neq\emptyset$, then
    $2^{k-3}\leq 2^{j}\leq 2^{k+3}$ by (1) in Lemma~\ref{geomannuli}.
  \item If $k\in\N\setminus\{1\}$, $j\in\{0,1\}$, and
    $A_2(x,r,k)\cap A_2(z,r,j)\neq\emptyset$, then
    $A_2(x,r,k)\cap B_z(2^jr)\neq\emptyset$
    and so $k\leq4$ by (2) in Lemma~\ref{geomannuli}.
  \item If $k\in\{0,1\}$, $j\in\N_0$, and
    $A_2(x,r,k)\cap A_2(z,r,j)\neq\emptyset$, then
    $B_x(2^kr)\cap A_2(z,r,j)\neq\emptyset$
    and so $j\leq4$ by (2) in Lemma~\ref{geomannuli}.
  \end{enumerate}
  Using this observation and H\"older's inequality, we can estimate
  \begin{align*}
    \|\one_{A_2(x,r,k)}\one_{A_2(z,r,j)}f\|_p
    \lesssim_{d,p,\beta} \one_{\{|j-k|\leq4\}}(2^kr)^{d(\frac1p-\frac12)}\|f\|_2\,,
    \quad |x-z|\leq2r
  \end{align*}
  and deduce (using $g(2^j)\sim_\beta 2^{-j\beta}$)
  \begin{align*}
    & \|\one_{B_x(r)}T_r\one_{A_2(x,r,k)}f\|_{p'}\\
    & \quad \lesssim C_\DG\, r^{-\frac d2}\sum_{j\geq0}g(2^j)2^{\frac{jd}{p'}}2^{kd(\frac1p-\frac12)}\one_{\{|j-k|\leq4\}}\Big(\int\limits_{|x-z|\leq2r}dz \Big)^{\frac{1}{p'}}\|f\|_2\\
    & \quad \lesssim_{p,d,\beta} C_\DG\, r^{-d(\frac12-\frac{1}{p'})}g(2^k)2^{kd/2}\|f\|_2\,.
  \end{align*}
  This concludes the proof of \eqref{eq:hypercontractivedual}.

  The proof of \eqref{eq:hypercontractivedual2} is similar but
  uses also (1) of Lemma~\ref{collectionmaximal}, i.e.,
  $(N_{2,r}f)(x)\lesssim_p (N_{p',r}f)(x)$.
  By the above reasoning and
  $\|\one_{A_2(x,r,k)}\one_{A_2(z,r,j)}f\|_{p}\leq \one_{\{|j-k|\leq4\}}\|f\|_p$
  for $|x-z|<2r$, we obtain
  \begin{align*}
    & \|\one_{B_x(r)}T_r\one_{A_2(x,r,k)}f\|_{2}
    \lesssim \|N_{2,r}\one_{B_x(r)}T_r\one_{A_2(x,r,k)}f\|_{2}
    \lesssim_p \|N_{p',r}\one_{B_x(r)}T_r\one_{A_2(x,r,k)}f\|_{2}\\
    & \quad = \left( \int_{\R^d}dz \left( r^{-\frac{d}{p'}}\|\one_{B_z(r)}\one_{B_x(r)}T_r \one_{A_2(x,r,k)}\sum_{j\geq0}\one_{A_2(z,r,j)}f\|_{p'} \right)^{2} \right)^{1/2}\\
    & \quad \leq C_\DG\, r^{-\frac{d}{p}}\sum_{j\geq0}g(2^j)2^{\frac{jd}{p'}} \left(\int\limits_{|x-z|\leq2r}dz\ \|\one_{A_2(x,r,k)}\one_{A_2(z,r,j)}f\|_{p}^{2}\right)^{1/2}\\
    & \quad \lesssim C_\DG\, r^{-d(\frac1p-\frac12)}\|f\|_p \sum_{j\geq0}g(2^j)2^{\frac{jd}{p'}}\one_{\{|j-k|\leq4\}}
      \lesssim C_\DG\, r^{-d(\frac1p-\frac12)}g(2^k)2^{\frac{kd}{p'}}\|f\|_p\,.
  \end{align*}
  This concludes the proof of Proposition~\ref{hypercontractivedual}.
\end{proof}

Combining Proposition~\ref{hypercontractivedual}
with Proposition~\ref{hypercontractive} yields the
following variant of \eqref{eq:dualityheat}, which does not need that
$T_r$ has the semigroup property.

\begin{corollary}
  \label{hypercontractivedualcor}
  Let $p\in[1,2]$, $r>0$, and $(T_r)_{r>0}$ be a family of linear
  operators that satisfy the dual dyadic $(p,p',p')$ Davies--Gaffney
  estimate \eqref{eq:gge3} (Definition \ref{ggerestricted}).
  Then
  \begin{subequations}
    \begin{align}
      \label{eq:hypercontractivedualcor1}
      \|T_r\|_{2\to p'}
      \lesssim_{d,\beta,p} C_\DG\, r^{-d\left(\frac12-\frac{1}{p'}\right)}
    \end{align}
    and
    \begin{align}
      \label{eq:hypercontractivedualcor2}
      \|T_r\|_{p\to 2}
      \lesssim_{d,\beta,p} C_\DG\, r^{-d\left(\frac1p-\frac12\right)}
    \end{align}
  \end{subequations}
  for all $x\in\R^d$.
\end{corollary}

We now show that restricted dyadic $(p,2,p')$ and $(p,p',p')$ estimates
imply (restricted) dyadic $(p,p,p')$ estimates and $L^p$ boundedness of $T_r$.

\begin{corollary}
  \label{lpboundedness}
  Let $p\in[1,2]$, $r>0$, and $(T_r)_{r>0}$ be a family of linear 
  operators that satisfy the restricted dyadic $(p,2,p')$ or $(p,p',p')$
  Davies--Gaffney estimate \eqref{eq:gge3} (Definition \ref{ggerestricted}).
  Then $T_r$ satisfies the restricted dyadic $(p,p,p')$ Davies--Gaffney estimate
  \begin{align}
    \label{eq:lpboundednessgge}
    \|\one_{B_x(r)}T_r\one_{A_2(x,r,k)}\|_{p\to p}
    \lesssim_{d,\beta,p} C_\DG\, 2^{\frac{kd}{p'}}g(2^k)\,,
    \quad x\in\R^d,k\in\N_0\,.
  \end{align}
  Moreover,
  \begin{align}
    \label{eq:lpboundednessgge2}
    \|T_r\|_{p\to p} = \|(T_r)^*\|_{p'\to p'} \lesssim_{d,\beta,p} C_\DG\,.
  \end{align}
\end{corollary}

\begin{proof}
  By Proposition~\ref{hypercontractivedual} it suffices to assume that
  $(T_r)_{r>0}$ satisfies the restricted $(p,2,p')$ Davies--Gaffney estimates.
  By H\"older's inequality and \eqref{eq:hypercontractivedual2} we have
  \begin{align*}
    \|\one_{B_x(r)}T_r\one_{A_2(x,r,k)}f\|_{p}
    & \leq |B_x(r)|^{\frac{2-p}{2p}} \|\one_{B_x(r)}T_r\one_{A_2(x,r,k)}f\|_2\\
    & \lesssim_{d,\beta,p} C_\DG\, r^{d\frac{2-p}{2p} - d\left(\frac1p-\frac12\right)} 2^{\frac{kd}{p'}}g(2^k)\|f\|_p
  \end{align*}
  for any $f\in L^p$, $x\in\R^d$, and $k\in\N_0$. This proves
  \eqref{eq:lpboundednessgge}.
  Formula \eqref{eq:lpboundednessgge2} follows from
  \eqref{eq:lpboundednessgge} and Proposition~\ref{hypercontractive}.
  This concludes the proof.
\end{proof}

\section*{Acknowledgments}
I wish to thank Volker Bach and Jean--Claude Cuenin for valuable discussions,
and an anonymous referee for many helpful remarks and suggestions. 


\def\cprime{$'$}

\end{document}